\documentclass[10pt]{article}
\usepackage{amsmath,amssymb}
\usepackage{enumerate}
\usepackage[dvipsnames]{xcolor}

\allowdisplaybreaks[1]

\newtheorem{theorem}{Theorem}[section]

\newtheorem{lemma}[theorem]{Lemma}
\newtheorem{proposition}[theorem]{Proposition}

\newtheorem{definition}[theorem]{Definition}
\newtheorem{remark}[theorem]{Remark}
\newtheorem{example}[theorem]{Example}
\newtheorem{question}[theorem]{Question}

\newcommand{\Z}{\mathbb{Z}}

\renewcommand{\ker}{\operatorname{Ker}}
\newcommand{\id}{\operatorname{id}}

\newcommand{\Sym}{\operatorname{Sym}}
\newcommand{\aut}{\operatorname{Aut}}

\newcommand{\soc}{\operatorname{Soc}}

\newcommand{\Aut}{\operatorname{Aut}}

\newcommand{\Ret}{\operatorname{Ret}}
\newcommand{\gr}{\operatorname{gr}}

\newenvironment{proof}{\par\noindent{\bf Proof.}}{$\qed$\par\bigskip}
\newcommand{\qed}{\enspace\vrule  height6pt  width4pt  depth2pt}
\usepackage{color}

\begin{document}
\title{Constructing finite simple solutions of the Yang-Baxter equation\thanks{The first author was partially
supported by the grants MINECO-FEDER  MTM2017-83487-P and AGAUR
2017SGR1725 (Spain).  The second author is supported by the
National Science Centre  grant. 2016/23/B/ST1/01045 (Poland). 2010
MSC: Primary 16T25, 20B15, 20F16. Keywords: Yang-Baxter equation,
set-theoretic solution, primitive group, brace.} }
\author{F. Ced\'o \and J. Okni\'{n}ski}
\date{}

\maketitle

\begin{abstract}
We study involutive non-degenerate set-theoretic solutions $(X,r)$
of the Yang-Baxter equation on a finite set $X$. The emphasis is
on the case where $(X,r)$ is indecomposable, so  the associated
permutation group ${\mathcal G}(X,r)$ acts transitively on $X$.
One of the major problems is to determine how such solutions are
built from the imprimitivity blocks; and also how to characterize
these blocks. We focus on the case of so called simple solutions,
which are of key importance. Several infinite families of such
solutions are constructed for the first time. In particular, a
broad class of simple solutions of order $p^2$, for any prime $p$,
is completely characterized.
\end{abstract}

\section{Introduction}

The Yang-Baxter equation appeared independently in a paper of Yang
\cite{Yang} an in a paper of Baxter \cite{Baxter} and then it
became one of the important equations of mathematical physics. It
also appears in the foundations of quantum groups and Hopf
algebras (see for example \cite{BrownGoodearl,K}). Recall that a
solution of the Yang-Baxter equation is an automorphism $R$ of a
vector space $V\otimes V$,  such that
$$R_{12}R_{23}R_{12} = R_{23}R_{12}R_{23},$$
where  $R_{ij}$ denotes the
map $V\otimes V\otimes V \rightarrow V\otimes V\otimes V$ acting as
$R$ on the $(i,j)$ tensor factors and as the identity on the
remaining factor. A difficult and important open problem is to find all
the solutions of the Yang-Baxter
equation. Drinfeld in
\cite{drinfeld} suggested the study of the set-theoretic solutions
of the Yang-Baxter equation, these are pairs $(X,r)$, where $X$ is a non-empty set and $r : X
\times  X \rightarrow X \times X$ is a bijective map
such that
$$r_{12}r_{23}r_{12} = r_{23}r_{12}r_{23},$$
where  $r_{ij}$ denotes the map $X \times X \times X \rightarrow X
\times X \times X$ acting as $r$ on the $(i,j)$ components and as
the identity on the remaining component.

The papers of Gateva-Ivanova and Van den Bergh \cite{GIVdB}, and
Etingof, Schedler and Soloviev \cite{ESS} introduced a special
class of so called involutive non-degenerate set-theoretic
solutions of the Yang-Baxter equation. And the study of this
important class of solutions has been exploding in the last twenty
years, see for example
\cite{BCV,CCP,CPR,CJO,CJOComm,GI18,GIC,Rump1,V}. Recall that a
set-theoretic solution $r : X \times  X \rightarrow X \times  X$
of the Yang-Baxter equation, written in the form $r(x,y) =
(\sigma_x(y),\gamma_y(x))$, for $x,y \in X$, is involutive if
$r^{2} = \id$, and it is non-degenerate if $\sigma_x$ and
$\gamma_x$ are bijective maps from $X$ to $X$, for all $ x \in X$.

In \cite{ESS} and \cite{GIVdB} a number of very fruitful algebraic
structures were introduced to study this class of solutions: the
structure group, the structure monoid  and the structure algebra
associated to a solution. More recently, in \cite{R07}, Rump
introduced a new algebraic structure, called (left) brace, in this
context. This attracted a lot of attention (see for example
\cite{B16,Brz19,CCS,DS,GI18,Smok} and the references in these
papers). And this allowed to construct several new families of
solutions and led to a discovery of several unexpected connections
to a variety of other areas \cite{CedoSurvey, RumpSurvey}.

In the study of involutive non-degenerate set-theoretic solutions
of the Yang-Baxter equation, a very fruitful approach is based on
the notions of indecomposable solutions and irretractable
solutions \cite{ESS}. Roughly speaking, the underlying idea is to
show that several classes of solutions come from solutions of
smaller cardinality.

Every (involutive non-degenerate set-theoretic) solution $(X,r)$
of the YBE is equipped with a permutation group ${\mathcal
G}(X,r)$ acting on the set $X$. It has been expected for a long
time that a possible classification of solutions $(X,r)$ of the
YBE would have to be based on the associated  groups $\mathcal G
(X,r)\subseteq \Sym_{X}$. Therefore, it is natural to anticipate
that some aspects of the theory of permutation groups will be
crucial in this context. However, so far, mainly transitivity of
the group $\mathcal G (X,r)$ (equivalent to indecomposability of
the solution $(X,r)$) has been explored. This has led to important
results on decomposability of solutions, but on the other hand it
turned out that in general indecomposable solutions are difficult
to construct and classify
\cite{CCP,CPR,Jedl-Pilit-Zam,rump2020,SmokSmok}. In particular, a
fundamental result of Rump \cite{Rump1} shows that all finite
square-free involutive non-degenerate set-theoretic solutions
$(X,r)$, with $|X|>1$, are decomposable. However, this is no
longer true in full generality.

The second approach, based on the retract relation, allowed to
introduce the class of multipermutation solutions and to define
the multipermutation level, which is a measure of their
complexity. It is related to the so called structure group
$G(X,r)$ (which is  a Bieberbach group if $X$ is finite)
associated to each solution $(X,r)$. Certain positive results in
this direction were obtained in \cite{CJO,CJOComm,GIC} and later
in \cite{BCV} it was shown that finite multipermutation solutions
coincide with solutions whose structure groups are poly-${\mathbb
Z}$-groups. However, this approach also fails in full generality
because there exist solutions that are not retractable.

In a talk \cite{BB} during a workshop in Oberwolfach
\cite{Ballesteros}, Ballester-Bolinches asked for a description of
all finite primitive solutions, i.e. solutions such that $\mathcal
G (X,r)$ acts on $X$ as a primitive permutation group. It was
recently shown in \cite{CJOprimit} that there is only one class of
such solutions of cardinality $>1$. Namely, every finite primitive
solution $(X,r)$ is of prime order, i.e. $|X|=p$ is a prime
number, and it is a so called permutation solution determined by a
cyclic permutation of length $p$. This opens a new perspective on
the classification problem of all finite solutions because,
roughly speaking, this shows that every finite solution $(X,r)$
which is not of the above form is built on an information coming
from its imprimitivity blocks, which are sets of smaller
cardinality. And now the challenge is to understand the structure
of possible imprimitivity blocks of an arbitrary solution and to
see how it can be built from its imprimitivity blocks. On the
other hand, recent results from \cite{CCP} show that every finite
indecomposable solution of the YBE is a so called dynamical
extension (introduced by Vendramin in \cite{V}) of a simple
solution. The latter are defined as solutions $(X,r)$ that do not
admit any nontrivial epimorphism of solutions $(X,r)\rightarrow
(Y,r')$ (meaning that $1<|Y|<|X|$). For example, all finite
primitive solutions are simple. This has been, till now, the only
known infinite family of finite simple solutions. So the challenge
is to construct and classify all simple solutions. This is our
main motivation in this paper.

The paper is organized as follows. In Section \ref{prelim} we
introduce the necessary background on involutive non-degenerate
set-theoretic solutions of the Yang-Baxter equation, left braces
and on cycle-sets, another structure introduced by Rump in
\cite{Rump1} to study this class of solutions. In Section
\ref{indecomp}, we expose the recent advances in the study of
indecomposable solutions and we introduce simple solutions. In
Section \ref{examples} some properties of simple solutions are
studied. For example, we prove that finite simple solutions are
indecomposable and irretractable (if their order is not a prime).
Then we focus on solutions of a special but quite general
type (Proposition~\ref{sec6newind}), which allows us to construct
infinite families of finite simple solutions of the Yang-Baxter
equation, in Theorems~\ref{newexample2}, \ref{newexample2bis} and
\ref{simplemn}. In particular, we show that for every positive
integer $n$, distinct primes $p_1,\dots ,p_n$ and positive
integers $m_1,\dots ,m_n$, such that $\sum_{i=1}^{n}m_i>n$, there
exists a simple solution of the Yang-Baxter equation of
cardinality $p_1^{m_1}\cdots p_n^{m_n}$. In Section \ref{sec5} a
wide class of simple solutions of cardinality $p^2$, for every
prime $p$, is characterized; see Theorems~\ref{newexample} and
\ref{thmmain}. In Section \ref{section6} we give an alternative
construction of a family of simple solutions of square cardinality
introduced in Section \ref{examples}, in terms of the asymmetric
product of left braces. Finally, in Section \ref{sec7} some open
questions related to simple solutions are proposed.

\section{Preliminaries} \label{prelim}
Let $X$ be a non-empty set and  let  $r:X\times X \rightarrow
X\times X$ be a map. For $x,y\in X$ we put $r(x,y) =(\sigma_x (y),
\gamma_y (x))$. Recall that $(X,r)$ is an involutive,
non-degenerate, set-theoretic solution of the Yang-Baxter equation
if $r^2=\id$, all the maps $\sigma_x$ and $\gamma_y$ are bijective
maps from $X$ to itself and
  $$r_{12} r_{23} r_{12} =r_{23} r_{12} r_{23},$$
where $r_{12}=r\times \id_X$ and $r_{23}=\id_X\times r$ are maps
from $X^3$ to itself. Because $r^{2}=\id$, one easily verifies that
$\gamma_y(x)=\sigma^{-1}_{\sigma_x(y)}(x)$, for all $x,y\in X$ (see
for example \cite[Proposition~1.6]{ESS}).

\bigskip
\noindent {\bf Convention.} Throughout the paper a solution of the
YBE will mean an involutive, non-degenerate, set-theoretic
solution
of the Yang-Baxter equation.\\

To study this class of solutions of the YBE, Rump \cite{R07}
introduced the algebraic structure called a left brace. We recall
some essential background (see \cite{CedoSurvey} for details). A
left brace is a set $B$ with two binary operations, $+$ and
$\circ$, such that $(B,+)$ is an abelian group (the additive group
of $B$), $(B,\circ)$ is a group (the multiplicative group of $B$),
and for every $a,b,c\in B$,
 \begin{eqnarray} \label{braceeq}
  a\circ (b+c)+a&=&a\circ b+a\circ c.
 \end{eqnarray}
Note that if we denote by $0$ the neutral element of $(B,+)$ and by
$1$ the neutral element of $(B,\circ)$, then
$$1=1\circ (0+0)+1=1\circ 0+1\circ 0=0.$$
In any left brace $B$ there is an action $\lambda\colon
(B,\circ)\rightarrow \aut(B,+)$, called the lambda map of $B$,
defined by $\lambda(a)=\lambda_a$ and $\lambda_{a}(b)=-a+a\circ b$,
for $a,b\in B$. We shall write $a\circ b=ab$, for all $a,b\in B$.
A trivial brace is a left brace $B$ such that $ab=a+b$, for all
$a,b\in B$, i.e. all $\lambda_a=\id$. The socle of a left brace
$B$ is
$$\soc(B)=\{ a\in B\mid ab=a+b, \mbox{ for all
}b\in B \}.$$ Note that $\soc(B)=\ker(\lambda)$, and thus it is a
normal subgroup of the multiplicative group of $B$. The solution of
the YBE associated to a left brace $B$ is $(B,r_B)$, where
$r_B(a,b)=(\lambda_a(b),\lambda_{\lambda_a(b)}^{-1}(a))$, for all
$a,b\in B$ (see \cite[Lemma~2]{CJOComm}).

A left ideal of a left brace $B$ is a subgroup $L$ of the additive
group of $B$ such that $\lambda_a(b)\in L$, for all $b\in L$ and
all $a\in B$. An ideal of a left brace $B$ is a normal subgroup
$I$ of the multiplicative group of $B$ such that $\lambda_a(b)\in
I$, for all $b\in I$ and all $a\in B$. Note that
\begin{eqnarray}\label{addmult1}
ab^{-1}&=&a-\lambda_{ab^{-1}}(b)
\end{eqnarray}
 for all $a,b\in B$, and
    \begin{eqnarray} \label{addmult2}
     &&a-b=a+\lambda_{b}(b^{-1})= a\lambda_{a^{-1}}(\lambda_b(b^{-1}))= a\lambda_{a^{-1}b}(b^{-1}),
     \end{eqnarray}
for all $a,b\in B$. Hence, every left ideal $L$ of $B$ also is a
subgroup of the multiplicative group of $B$, and every  ideal $I$ of
a left brace $B$ also is a subgroup of the additive group of $B$. For example, it
is known that $\soc(B)$ is an ideal of the left brace $B$ (see
\cite[Proposition~7]{R07}).  Note that $B/I$ inherits a natural left brace structure.

Let $B$ be a left brace. We define another binary operation $*$ on $B$ by
$$a*b\colon =-a+ab-b=\lambda_a(b)-b,$$
for all $a,b\in B$. By \cite[Corollary of Proposition 6]{R07},
the subgroup $B^2=B*B=\gr(a*b\mid a,b\in B)_+$ of $(B,+)$ generated by all the elements of the form $a*b$ is an ideal of $B$.

Recall that if $(X,r)$ is a solution of the YBE, with
$r(x,y)=(\sigma_x(y),\gamma_y(x))$, then its structure group
$G(X,r)=\gr(x\in X\mid xy=\sigma_x(y)\gamma_y(x),\mbox{ for all
}x,y\in X)$ has a natural structure of a left brace such that
$\lambda_x(y)=\sigma_x(y)$, for all $x,y\in X$. The additive group
of $G(X,r)$ is the free abelian group with basis $X$. The
permutation group $\mathcal{G}(X,r)=\gr(\sigma_x\mid x\in X)$ of
$(X,r)$ is a subgroup of the symmetric group $\Sym_X$ on $X$.  The
map $x\mapsto \sigma_x$, from $X$ to $\mathcal{G}(X,r)$ extends to a
group homomorphism $\phi: G(X,r)\longrightarrow \mathcal{G}(X,r)$
and $\ker(\phi)=\soc(G(X,r))$. Hence there is a unique structure of
a left brace on $\mathcal{G}(X,r)$ such that $\phi$ is a homomorphism
of left braces, this is the natural structure of a left brace on
$\mathcal{G}(X,r)$.

\begin{lemma}\label{lambda}
\cite[Lemma 2.1]{CJOprimit}
Let $(X,r)$ be a solution of the YBE. Then
$\lambda_g(\sigma_x)=\sigma_{g(x)}$, for all $g\in \mathcal{G}(X,r)$
and all $x\in X$.
\end{lemma}

 The following result is well known (see \cite[Proposition
7]{R07}), but we give a new proof for the convenience of the
reader.
\begin{lemma}\label{known}
    Let $B$ be a left brace. Then $B/\soc(B)\cong \mathcal{G}(B,r_B)$ as left braces.
\end{lemma}
\begin{proof}  By the definition of the lambda map of $B$,
we have $$\lambda_a(b)\lambda^{-1}_{\lambda_a(b)}(a)=\lambda_a(b)+\lambda_{\lambda_a(b)}(\lambda^{-1}_{\lambda_a(b)}(a))=\lambda_a(b)+a=ab,$$
    for all $a,b\in B$.  Hence,
    by the definition of the group $G(B,r_B)$,
    there is a homomorphism of left braces
    $\varphi\colon G(B,r_B)\longrightarrow B$, such that $\varphi(a)=a$,
    for all $a\in B$.
    Let $g\in \varphi^{-1}(\soc(B))$. Then $\lambda_{\varphi
    (g)}(b) =b$ for $b\in B$. In $G(B,r_B)$ we
    have that $\lambda_g(b)\in B$. Hence
    $$b=\lambda_{\varphi
    (g)}(b)= \varphi(g)b-\varphi(g)=\varphi(g)\varphi(b)-\varphi(g)=\varphi(gb-g)=
    \varphi(\lambda_g(b))=\lambda_g(b),$$
    for all $b\in B$. Since the additive group of $G(B,r_B)$ is the free
    abelian group with basis $B$ and $\lambda_g\in \Aut(G(B,r_B),+)$, we get that
    $\lambda_g(h)=h$, for all $h\in G(B,r_B)$. Therefore
    $g \in \soc(G(B,r_B))$  and hence
    $\varphi^{-1}(\soc(B))=\soc(G(B,r_B))$ follows.
    Thus
    $$\mathcal{G}(B,r_B)\cong G(B,r_B)/\soc(G(B,r_B))\cong B/\soc(B),$$
    as left braces.
\end{proof}

Let $(X,r)$ and $(Y,s)$ be solutions of the YBE. We write
$r(x,y)=(\sigma_x(y),\gamma_y(x))$ and
$s(t,z)=(\sigma'_t(z),\gamma'_z(t))$, for all $x,y\in X$ and $t,z\in
Y$. A homomorphism of solutions $f\colon (X,r)\longrightarrow (Y,s)$
is a map $f\colon X\longrightarrow Y$ such that
$f(\sigma_x(y))=\sigma'_{f(x)}(f(y))$ and
$f(\gamma_y(x))=\gamma'_{f(y)}(f(x))$, for all $x,y\in X$. Since
$\gamma_y(x)=\sigma^{-1}_{\sigma_x(y)}(x)$ and
$\gamma'_z(t)=(\sigma')^{-1}_{\sigma'_t(z)}(t)$, it is clear that
$f$ is a homomorphism of solutions if and only if
$f(\sigma_x(y))=\sigma'_{f(x)}(f(y))$, for all $x,y\in X$.

Note that every homomorphism of solutions $f\colon (X,r)\longrightarrow (Y,s)$ extends to a homomorphism of left braces
$f\colon G(X,r)\longrightarrow G(Y,s)$ that we also denote by $f$, and induces a homomorphism of left braces $\bar f\colon \mathcal{G}(X,r)\longrightarrow\mathcal{G}(Y,s)$.

In \cite{ESS}, Etingof, Schedler and Soloviev introduced the retract
relation on solutions  $(X,r)$ of the YBE. This is the binary
relation $\sim$ on $X$ defined by $x\sim y$ if and only if
$\sigma_x=\sigma_y$. Then, $\sim$ is an equivalence relation and $r$
induces a solution $\overline{r}$ on $\overline{X}=X/{\sim}$. The
retract of the solution $(X,r)$ is
$\Ret(X,r)=(\overline{X},\overline{r})$. Note that the natural map
$f\colon X\longrightarrow \overline{X}:x\mapsto \bar x$ is an
epimorphism of solutions from $(X,r)$ onto $\Ret(X,r)$.

Define $\Ret^{1}(X,r)=\Ret(X,r)$ and for every positive integer $n$, $\Ret^{n+1}(X,r)=\Ret(\Ret^{n}(X,r))$.
The solution $(X,r)$ is   said to be  a multipermutation solution of level $n$ if $n$ is the smaller positive integer such that $\Ret^{n}(X,r)$ has cardinality $1$.
Recall that a solution $(X,r)$ is  said to be irretractable if
$\sigma_x\neq \sigma_y$ for all distinct elements $x,y\in X$, that is $(X,r)=\Ret(X,r)$,
otherwise the solution $(X,r)$ is retractable.

Another algebraic structure introduced by Rump \cite{Rump1} to study solutions of the YBE is
the structure of a left cycle set. Recall that a left cycle set is a pair $(X,\cdot)$ of a non-empty set $X$ and a binary operation $\cdot$ such that
$$(x\cdot y)\cdot(x\cdot z)=(y\cdot x)\cdot(y\cdot z),$$
for all $x,y,z\in X$, and the map $y\mapsto x\cdot y$ is a
bijective map from $X$ to itself for every $x\in X$. A left cycle
set $(X,\cdot)$ is non-degenerate if the map $x\mapsto x\cdot x$
is a bijective map from $X$ to itself. In \cite[Propositions 1 and
2]{Rump1} it is proven that there is a bijective correspondence
between non-degenerate left cycle sets and solutions of the YBE.
That is, if $(X,\cdot)$ is a non-degenerate left cycle set, and
$\sigma_x$ is the inverse of the map $y\mapsto x\cdot y$, then
$(X,r)$ is the corresponding solution of the YBE, where
$r(x,y)=(\sigma_x(y), \sigma_x(y)\cdot x)$, for all $x,y\in X$.

In \cite{V} Vendramin introduced the dynamical extension of a
cycle set. Let $I$ be a left cycle set and let $S$ be a non-empty
set. A map $\alpha\colon I\times I\times S \longrightarrow \Sym_S$
is a dynamical cocycle of $I$ with values in $S$ if, for all $i, j, k\in I$ and $r, s, t \in S$,
$$\alpha_{(i\cdot j,i\cdot k)}(\alpha_{(i,j)}(r, s),\alpha_{(i,k)}(r, t)) = \alpha_{(j\cdot i,j\cdot k)}(\alpha_{(j,i)}(s, r),\alpha_{(j,k)}(s, t)),$$
where $\alpha_{(i,j)}(r,s)\colon\!\!\! =\alpha(i,j,r)(s)$. The
dynamical extension of $I$ by $\alpha$ is the left cycle set
$S\times_{\alpha} I\colon\!\!\! = (S \times I,\cdot)$, where
$$(s, i) \cdot (t, j) \colon\!\!\! = (\alpha_{(i,j)}(s, t), i \cdot  j),$$
for all $i, j\in I$ and $s, t \in S$.

\section{Indecomposable and simple solutions} \label{indecomp}

Let $(X,r)$ be a solution of the YBE. We say that $(X,r)$ is indecomposable if $\mathcal{G}(X,r)$ acts transitively on $X$.

The following definition is due to Adolfo Ballester-Bolinches
\cite{BB}.
\begin{definition}  \label{primitive}
A finite solution $(X,r)$ of the YBE is said to be primitive if its
permutation group $\mathcal{G}(X,r)$ acts primitively on $X$.
\end{definition}

By \cite[Theorem 2.13]{ESS}, for each prime $p$, there is, up to isomorphism,  a unique
indecomposable solution $(X,r)$ of the YBE of cardinality $p$. In
this case,  $X=\Z/(p)$ and $\sigma_i(j)=j+1$, for all $i,j\in
\Z/(p)$. Thus $\mathcal{G}(X,r)\cong \Z/(p)$, and $(X,r)$ is
primitive and it is a multipermutation solution of level $1$. In
\cite{CJOprimit} it is proven that these are all the finite
primitive solutions of the YBE of
cardinality $>1$.

We say that a finite indecomposable solution $(X,r)$ of the
YBE has {\em primitive level} $k$ if $k$ is the biggest positive
integer such that there exist solutions $(X_1,r_1)=(X,r),\,
(X_2,r_2),\dots ,(X_k,r_k)$ and epimorphisms of solutions $p_{i+1}\colon
(X_i,r_i)\longrightarrow (X_{i+1},r_{i+1})$, with $|X_i|>|X_{i+1}|>1$, for $1\leq
i\leq k-1$, and $(X_k,r_k)$ is primitive.

\begin{question} Describe solutions of primitive level $2$. As said above, solutions of primitive level $1$ (in other words, primitive solutions) admit a very simple description.
\end{question}

In \cite[Theorems 5.3 and 5.4]{SmokSmok} Agata Smoktunowicz and
Alicja Smoktunowicz gave  a method to construct any finite
indecomposable solution of the YBE, assuming that one is able to
construct all finite left braces. Indeed, let $B$ be a finite left
brace and let $x\in B$. Consider the left subbrace $B(x)$
generated by $x$. Let $X=\{\lambda_a(x)\mid a\in B(x)\}$. Then
$(X,r)$ is indecomposable, where $r$ is the restriction of $r_B$
to $X\times X$, and $(B,r_B)$ is the solution associated to the
left brace $B$. Furthermore, every finite indecomposable solution
of the YBE is constructed in this way.
However, it is difficult to describe the
structure of $B(x)$ and the properties of the indecomposable
solution $(X,r)$. For example, it is unclear how one can choose
$B$ and $x\in B$ such that $(X,r)$ is irretractable.   In
\cite[Section 6]{CSV} it is proven that this method also is valid
for infinite indecomposable solutions.

In \cite{rump2020} Rump gave another method to construct
indecomposable solutions of the YBE. Let $(X,r)$ be an
indecomposable finite solution of the YBE. Let
$\mathcal{G}=\mathcal{G}(X,r)$ and let $x\in X$. Let
$c_{x,\mathcal{G}}\colon \mathcal{G}\times
\mathcal{G}\longrightarrow \mathcal{G}\times \mathcal{G}$ be the map
defined by
$$c_{x,\mathcal{G}}(g,h)=(h\sigma^{-1}_{g^{-1}(x)}, g\sigma_{\sigma_{g^{-1}(x)}h^{-1}(x)}),$$
for all $g,h\in \mathcal{G}$. Then $(\mathcal{G},c_{x,\mathcal{G}})$
is an indecomposable finite solution of the YBE. Furthermore, the
map $p\colon \mathcal{G}\longrightarrow X$ defined by
$p(g)=g^{-1}(x)$ is an epimorphism of solutions,
$\mathcal{G}(\mathcal{G},c_{x,\mathcal{G}})\cong \mathcal{G}$ and
for every indecomposable finite solution $(Y,s)$ of the YBE and
every epimorphism $q\colon Y\longrightarrow X$ of solutions such
that the induced homomorphism of groups
$\mathcal{G}(Y,s)\longrightarrow \mathcal{G}$ is an isomorphism,
there exists a unique epimorphism of solutions $q'\colon
\mathcal{G}\longrightarrow Y$ such that $p= qq'$. Rump uses cycle
sets and $p$ is called the universal covering. The fundamental group
of the cycle set corresponding to the solution $(X,r)$ is the
stabilizer of $x$, i.e.
$$\pi_1(X)=\pi_1(X,x)=\{g\in \mathcal{G}\mid g(x)=x\}.$$
Then Rump presents a method to construct all the finite
indecomposable solutions $(Z,t)$ of the YBE such that the natural
structure of left brace of $\mathcal{G}(Z,t)$ is isomorphic to the
left brace $\mathcal{G}$. Thus this yields another method to
construct all the finite indecomposable solutions, but again
assuming that one is able to construct all finite left braces.

We now introduce the key notion studied in this paper.
\begin{definition}
A solution $(X,r)$ of the YBE is simple if $|X|>1$ and for every epimorphism of $f:(X,r) \rightarrow (Y,s)$ of solutions either $f$ is an isomorphism or $|Y|=1$.
\end{definition}
\begin{remark}\label{remsimple}
In \cite{V}, Vendramin introduced finite simple cycle sets. His
definition does not coincide with the above definition of
simplicity, but for finite indecomposable cycle sets both
definitions coincide by \cite[Proposition 2]{CCP}.
\end{remark}

Suppose that $(X,r)$ is an indecomposable finite solution of the
YBE which is not simple. Thus there exists a solution $(Y,s)$  of
the YBE and an epimorphism $p\colon X\longrightarrow Y$ of
solutions such that $1<|Y|<|X|$. Note that $(Y,s)$ also is
indecomposable.  In
    \cite[Proposition 2]{CCP} Castelli, Catino and Pinto proved that
    the cycle set corresponding to $(X,r)$ is a dynamical extension of
    the cycle  set $(Y,\cdot)$ corresponding to $(Y,s)$  by a
    dynamical cocycle $\alpha$ of the cycle set $(Y,\cdot)$ with
    values in $p^{-1}(y)$, for some $y\in Y$.  We also say that the solution
    $(X,r)$ is a dynamical extension of the solution $(Y,s)$. By
    \cite[Lemma 1]{CCP}, $|p^{-1}(y)|=|p^{-1}(y')|$ for all $y,y'\in Y$,
    i.e. $p$ is a covering of the corresponding cycle sets.
 Furthermore, in the proof of \cite[Lemma 1]{CCP} it is
proved that $g(p^{-1}(p(x)))=p^{-1}(p(g(x)))$, for all $x\in X$  and
$g\in \mathcal{G}(X,r)$. Hence $\{ p^{-1}(y)\mid y\in Y\}$ is a set
of imprimitivity blocks of $X$ under the action of
$\mathcal{G}(X,r)$. For every $y\in Y$, let $H_y$ be the  following
subgroup of $\mathcal{G}(X,r)$
$$H_y=\{g\in \mathcal{G}(X,r)\mid g(p^{-1}(y))=p^{-1}(y)\}.$$
Since $(X,r)$ is indecomposable, $H_y$ acts transitively on
$p^{-1}(y)$ (see \cite[Theorem 7]{CCP}). In \cite{CCP} there are
some concrete examples of indecomposable solutions constructed in
this way (using dynamical extensions of indecomposable cycle sets).
\cite[Proposition 10]{CCP} gives a criterion to check whether a
dynamical extension of a cycle set is irretractable.

Since every finite indecomposable solution of the
YBE is a dynamical extension of a simple solution of the YBE,
another strategy to construct all the finite indecomposable
solutions is to construct all the finite simple solutions of the YBE
and to determine all the dynamical cocycles on every cycle set
corresponding to a finite simple solution of the YBE.

We define the composition
    length of a finite indecomposable solution $(X,r)$  as the biggest
    positive integer $k$ such that there exist solutions
    $(X_1,r_1)=(X,r),\, (X_2,r_2),\dots ,(X_k,r_k)$ and epimorphisms of solutions
    $p_{i+1}\colon (X_i,r_i)\longrightarrow (X_{i+1},r_{i+1})$, with
    $|X_i|>|X_{i+1}|$, for $1\leq i\leq k-1$, and $(X_k,r_k)$ is
    simple.

From a computer calculation one obtains that the
following two examples are the only  solutions on a set of
cardinality four that are indecomposable and irretractable.

\begin{example}  \label{examp1}
Let $X=\{1,2,3,4\}$. Define permutations
$$\sigma_1 = (2,3),\ \sigma_2= (1,4) , \ \sigma_3 = (1,2,4,3) ,\
\sigma_4 =(1,3,4,2)\in \Sym_X.$$ Then $(X,r)$ is a solution of the
YBE, with $r(x,y)=(\sigma_x(y),\sigma_{\sigma_x(y)}^{-1}(x))$, for
all $x,y\in X$. Notice that $\mathcal{G}(X,r)= \langle \sigma_1 ,
\sigma_2 , \sigma_3 , \sigma_4 \rangle $ is isomorphic to the
dihedral group of order $8$. This is an example of an
indecomposable and irretractable solution, see \cite{JO} or
Example~8.2.14 in \cite{JObook}; actually the first known example
of a solution whose structure group is not a poly-$\mathbb{Z}$
group. It is clear that $\mathcal{G}(X,r)$ acts transitively on
$X$ and $X_{1} =\{ 1,4\}$ and $X_{2} = \{2,3\}$ form imprimitivity
blocks for the action of the group $\mathcal{G}(X,r)$ on $X$.
\end{example}

 The second example is as follows.

\begin{example}  \label{examp2}
Let $X=\{ 1,2,3,4\}$. Let
$$\sigma_1 = (1,2), \  \sigma_2 = (3,1,4,2), \  \sigma_3 =
(2,4,1,3), \  \sigma_4 =(3,4).$$ It is easy to check that  $(X,r)$
is an indecomposable and irretractable solution of the YBE, with
$r(x,y)=(\sigma_x(y),\sigma_{\sigma_x(y)}^{-1}(x))$, for all
$x,y\in X$. Also in this case
$\mathcal{G}(X,r)\cong D_4$. The imprimitivity blocks for the action of $\mathcal{G}(X,r)$ on $X$ are
$\{1,2\}$ and $\{3,4\}$.   This solution is not isomorphic to the
first one because we have $\sigma_i(i)\neq i$, for all $i\in\{
1,2,3,4\}$, while this is not true in Example~\ref{examp1}.
\end{example}

In fact, we will see that these two solutions of cardinality four
are the only simple solutions of the YBE of this cardinality (see Remark \ref{4simple}).

\section{Examples of simple solutions}  \label{examples}

 Note that every indecomposable solution of the YBE of prime
cardinality is simple. And these are the finite primitive
solutions of cardinality bigger than $1$.

In this section we construct new examples of finite simple solutions of the YBE.
First, we study some properties of this class of solutions.

\begin{lemma}\label{1}
    Let $(X,r)$ be  a simple solution of the YBE. If $|X|>2$ then $(X,r)$ is indecomposable.
\end{lemma}

\begin{proof}
    Suppose that $(X,r)$ is decomposable.
Then $X$ is a disjoint union $X=X_1 \cup X_2$ of two non-empty
subsets $X_1$ and $X_2$  of $X$ such that $r(X_i \times X_i) =X_i
\times X_i$, for $i=1,2$. Note that then
    $r(X_1 \times X_2) = X_2 \times X_1$ and
$r(X_2 \times X_1) =X_1 \times X_2$. Let $Y=\{ 1,2\}$ and $s:Y\times
Y \rightarrow Y\times Y$ be defined by $s(i,j)=(j,i)$, for all
$i,j\in Y$. Thus $(Y,s)$ is the trivial solution. Let
$f:X\rightarrow Y$ be the map defined by $f(x_i)=i$ if $x_i\in X_i$.
Clearly $f$ is an epimorphism of solutions and it is not an
isomorphism because $|X|>2$. But, by assumption, $(X,r)$ is simple,
a contradiction. Therefore $(X,r)$ is indecomposable.
\end{proof}

As mentioned earlier (see \cite{ESS}),
 if $(X,r)$ is an indecomposable solution  of the YBE and  $|X|$  is a prime, then
it is a multipermutation solution of level 1 (in particular it is retractable).

\begin{proposition}\label{irretract}
    Let $(X,r)$ be a finite simple solution of the YBE. If $|X|$ is not prime, then $(X,r)$ is irretractable.
\end{proposition}

\begin{proof}
    Suppose that $(X,r)$ is retractable. Consider the natural projection $\pi : (X,r) \rightarrow \Ret (X,r)$.
    Since, $\pi$ is an epimorphism of solutions and because it is not an isomorphism and $(X,r)$ is simple,
    we have that the cardinality of $\Ret (X,r)$ is 1.  Again write $r(x,y)=(\sigma_x (y), \sigma^{-1}_{\sigma_x (y)} (x))$ for $x,y\in X$.
    Then we have that
    $$\sigma_x =\sigma_y$$
    for all $x,y\in X$. Thus $r(x,y)=(\sigma(y),\sigma^{-1}(x))$, for some permutation $\sigma\in \Sym_{X}$,
    that is $(X,r)$ is a permutation solution, introduced by Lyubashenko; see \cite{drinfeld}. By Lemma~\ref{1}, since $|X|>2$, $(X,r)$ is indecomposable.
    By \cite[p. 184]{ESS}, $\sigma$ is a cycle of length $|X|$, that is $\sigma=( x_1, \ldots , x_n)$ with $n=|X|$.
    Since $n$ is not prime,
    there exist integers $1<d,m<n$ such that $n=dm$. Let $Y=\Z/(d)$  and let
    $s:Y\times Y \rightarrow Y\times Y$ be the map defined by $s(i,j)=(j+1,i-1)$ for all $i,j\in Y$.
    Then $(Y,s)$ is a solution of the YBE (a permutation solution). Let $f:X\rightarrow Y$ be the map defined by $f(x_i) =i \; (\in \Z/(d))$, for all $i=1, \ldots , n$.
    Note that
    $$f(\sigma_{x_{i}}(x_j)) =\left\{ \begin{array}{ll}
    f(x_{j+1}) & \; \mbox{if } j<n\\
    1 & \mbox{ if } j=n
    \end{array} \right. .
    $$
Hence $f$ is an epimorphism of solutions and $1<|Y|<|X|$, a
contradiction. Therefore $(X,r)$ is irretractable.
\end{proof}

Let $(X,r)$ be a finite simple solution such that $|X|$ is
not prime. By Lemma
\ref{1} and Proposition \ref{irretract}, we know  that $(X,r)$ is
indecomposable and irretractable. Consider the permutation group
$\mathcal{G}=\mathcal{G} (X,r)$. Then the map $x\mapsto \sigma_x$
is an injective morphism of solutions from $(X,r)$ to
$(\mathcal{G},r_{\mathcal{G}})$. Furthermore, by Lemma
\ref{lambda}, $\sigma(X)=\{\sigma_x\mid x\in X\}$ is an orbit by
the action of the lambda map in the left brace $\mathcal{G}$ that
generates the  multiplicative (and the additive) group of the left
brace $\mathcal{G}$.

\begin{proposition}   \label{idealI}
    With the above conditions, if $I$ is a minimal nonzero ideal of $\mathcal{G}$, then $\mathcal{G}/I$ is a trivial cyclic brace,
   and $I=\mathcal{G}^2=\mathcal{G}*\mathcal{G}=\langle \sigma_x-\sigma_y\mid x,y\in X\rangle_+$. Furthermore $\soc(\mathcal{G})=\{ 0\}$.
\end{proposition}

\begin{proof}
       Let $I$ be a minimal nonzero ideal of $\mathcal{G}$. Let $\pi\colon \mathcal{G}\longrightarrow \mathcal{G}/I$ be the natural map.
       Let $g\in I$ be a nonzero element.  By Lemma \ref{lambda}, $\lambda_g(\sigma_x)=\sigma_{g(x)}$, for all $x\in X$. Hence
    $$\pi(\sigma_{g(x)})=\pi(\lambda_g(\sigma_x))=\lambda_{\pi(g)}(\pi(\sigma_x))=\lambda_0(\pi(\sigma_x))=\pi(\sigma_x).$$
    Thus the restriction $\pi|_{\sigma(X)}$ of $\pi$ induces an epimorphism $x\mapsto\pi(\sigma_x)$ of solutions from $(X,r)$ to $(\pi(\sigma(X)),s)$, where
    $$s(\pi(\sigma_x),\pi(\sigma_y))=(\pi(\sigma_{\sigma_x(y)}),\pi(\sigma_{\sigma^{-1}_{\sigma_x(y)}(x)})).$$
    Since $g\neq 0$, there exists $x\in X$ such that $x\neq g(x)$. Furthermore,
    since $\pi(\sigma_x)=\pi(\sigma_{g(x)})$ and $(X,r)$ is simple, we have that $\pi(\sigma_x)=\pi(\sigma_y)$, for all $x,y\in X$. Hence
    $$\langle \sigma_x-\sigma_y\mid x,y\in X\rangle_+\subseteq I.$$
    Note that
    \begin{eqnarray*}
        \mathcal{G}^2&=&\langle g*h \mid g,h\in \mathcal{G}\rangle_+\\
        &=&\langle \lambda_g(h)-h \mid g,h\in \mathcal{G}\rangle_+\\
        &=&\langle \lambda_g(\sigma_x)-\sigma_x \mid g\in \mathcal{G},\, x\in X\rangle_+\quad\mbox{ (since $\mathcal{G}=\langle\sigma_x\mid x\in X\rangle_+$)}\\
        &=&\langle \sigma_{g(x)}-\sigma_x \mid g\in \mathcal{G},\, x\in X\rangle_+ \quad\mbox{ (by the above)}\\
        &=&\langle \sigma_{y}-\sigma_x \mid  x,y\in X\rangle_+, \quad\mbox{ (since $(X,r)$ is indecomposable)}
    \end{eqnarray*}
    By the minimality of $I$ we get $I=\mathcal{G}^2=\langle \sigma_x-\sigma_y\mid x,y\in X\rangle_+$ .
    Now it is easy to see that if $x\in X$,  $\mathcal{G}/I=\langle \pi(\sigma_x)\rangle_+=\langle \pi(\sigma_x)\rangle$ is the trivial brace of the cyclic group $\langle \pi(\sigma_x)\rangle$.
    Finally let $g\in \soc(\mathcal{G})$. By Lemma \ref{lambda}, we have that
    $$\sigma_{g(x)}=\lambda_g(\sigma_x)=\sigma_x,$$
    for all $x\in X$. Since $(X,r)$ is irretractable, we get that $g(x)=x$ for all $x\in X$. This proves that $\soc(\mathcal{G})=\{ 0\}$, and the result follows.
\end{proof}

Of course, nontrivial finite simple left braces $B$ satisfy $B=B^2$
and $\soc(B)=\{ 0\}$.  There are several constructions of finite simple left braces,
see for example \cite{CJOabund} and the references therein, but we do not know the answer to the following questions.

\begin{question}\label{questionorbit}
Is there a nontrivial finite simple left brace $B$ with an orbit $X$ by the action of the lambda map such that $B=\langle X\rangle_+$?
\end{question}

\begin{question}  \label{question_simple}
Is there a finite simple solution $(X,r)$ of the YBE such that $\mathcal{G}(X,r)$ is a nontrivial simple left brace?
\end{question}

Note that an affirmative answer to Question \ref{question_simple}
implies an affirmative answer to Question \ref{questionorbit}.

We will study the simple solutions $(X,r)$ of the YBE, where
$X=Y\times Z$ and
$r(x,y)=(\sigma_x(y),\sigma^{-1}_{\sigma_x(y)}(x))$, for all
$x,y\in X$ and
$$\sigma_{(i,j)}(k,l) =(\sigma_{j}(k),d_{i,\sigma_j(k)}(l)),$$
for all $i,k\in Y$ and $j,l\in Z$. Note that in this case the sets
$X_i=\{ (i,j)\mid j\in Z\}$, for $i\in Y$, are imprimitivity
blocks for the action of $\mathcal{G}(X,r)$ on $X$, and we are
assuming that $\sigma_{i,j}X_k=X_{\sigma_j(k)}$ is independent of
$i$. We also simplify the second component of
$\sigma_{(i,j)}(k,l)$, imposing a link between $j$ and $k$, which
is $\sigma_j(k)$.

\begin{proposition}\label{sec6newind}
    Let $Y,Z$ be finite non-empty sets such that $|Y|,|Z|>1$. Let $X=Y\times Z$. Let $r\colon X\times X\longrightarrow X\times X$ be a map and write $r(x,y)=(\sigma_x(y),\gamma_y(x))$.
    Assume that $\sigma_{(i,j)} (k,l)=(\sigma_j(k),d_{i,\sigma_j(k)}(l))$,  for all $i,k\in Y$ and all $j,l\in Z$.
    Then $(X,r)$ is an indecomposable and irretractable solution of the YBE  if and only if $\sigma_j\in \Sym_Y$, $d_{i,k}\in \Sym_Z$,
    the permutation  subgroups
    $F=\langle\sigma_j\mid j\in Z\rangle\subseteq\Sym_Y$ and $W=\langle d_{i,k}\mid i,k\in Y\rangle\subseteq\Sym_Z$ are   transitive,
    \begin{enumerate}
        \item   $ r((i,j),(k,l)) =
        ((\sigma_j (k),d_{i,\sigma_j(k)}(l)) , (\sigma^{-1}_{d_{i,\sigma_j(k)}(l)}(i),d^{-1}_{\sigma_j(k),i}(j)) )$,
        \item $\sigma_j\circ\sigma_{d^{-1}_{i,k}(l)}=\sigma_l\circ\sigma_{d^{-1}_{k,i}(j)}$,
        \item $d_{i,k}=d_{k,i}$,
        \item $d_{i,w} \circ  d_{ \sigma^{-1}_{j}(k),\sigma^{-1}_j(w)}=d_{k,w } \circ  d_{\sigma^{-1}_{l}(i), \sigma^{-1}_l(w)},$
    \end{enumerate}
    for all $i,k,w\in Y$ and $j,l\in Z$, and
    \begin{itemize}
        \item[(i)] $\sigma_{j}\neq \sigma_{l}$ for $j\neq l$,
        \item[(ii)] if $d_{i,k}=    d_{i',k}$ for every $k$, then $i=i'$.
    \end{itemize}
\end{proposition}
\begin{proof}  First, we shall see that $(X,r)$ is a solution of the YBE if and only if
    $\sigma_j\in \Sym_Y$, $d_{i,k}\in \Sym_Z$, conditions 1. and 2. are satisfied
    and
    \begin{itemize}
        \item[(4')] $d_{i,\sigma_j\sigma_{d^{-1}_{i,k}(l)}(u)} \circ  d_{ \sigma^{-1}_{j}(k),\sigma_{d^{-1}_{i,k}(l)}(u)}
            =d_{k,\sigma_l\sigma_{d^{-1}_{k,i}(j)}(u) } \circ  d_{\sigma^{-1}_{l}(i), \sigma_{d^{-1}_{k,i}(l)}(u)},$
    \end{itemize}
    for all $i,k,u\in Y$ and $j,l\in Z$.

 We know from \cite{CJOComm} that $(X,r)$ is a solution of the YBE if and only if $\sigma_{(i,j)}$ is bijective for all $(i,j)\in X$ and the following conditions hold:
\begin{itemize}
    \item[(a)] $\gamma_{y}(x)=\sigma^{-1}_{\sigma_x(y)}(x)$ for all $x,y\in X$.
    \item[(b)] $\sigma_{x}\sigma_{\sigma^{-1}_{x}(y)}=\sigma_{y}\sigma_{\sigma^{-1}_{y}(x)}$, for all $x,y\in X$.
\end{itemize}
Clearly $\sigma_{(i,j)}$ is bijective if and only if $\sigma_j$ and $d_{i,\sigma_j(k)}$ are bijective for all $k\in Y$.
Condition (a) is equivalent to  $r((i,j),(k,l))=(\sigma_{(i,j)}(k,l), \sigma^{-1}_{\sigma_{(i,j)}(k,l)}(i,j))$, and by the definition of the maps $\sigma_j$ and $d_{i,k}$, we have
\begin{eqnarray*}r((i,j), (k,l))&=&(\sigma_{(i,j)}(k,l), \sigma^{-1}_{\sigma_{(i,j)}(k,l)}(i,j))\\
    &=& ((\sigma_j (k), d_{i,\sigma_j(k)}(l)) , (\sigma^{-1}_{(\sigma_j (k), d_{i,\sigma_j(k)}(l))}(i,j))\\
    &=&((\sigma_j (k),d_{i,\sigma_j(k)}(l)) , (\sigma^{-1}_{d_{i,\sigma_j(k)}(l)}(i),d^{-1}_{\sigma_j(k),i}(j)) ).
\end{eqnarray*}
Hence, condition (a) is equivalent to condition 1.

Note that
\begin{eqnarray*}
    \sigma_{(i,j)} \sigma_{\sigma^{-1}_{(i,j)}(k,l)} (u,v) &=&
    \sigma_{(i,j)} \sigma_{(\sigma^{-1}_{j}(k), d^{-1}_{i,k}(l))} (u,v)\\
    &=& \sigma_{(i,j)} \left(
    \sigma_{d^{-1}_{i,k}(l)}(u),
    d_{\sigma^{-1}_{j}(k), \; \sigma_{d^{-1}_{i,k }(l)}(u)}  (v)    \right)\\
    &=& (\sigma_{j} (\sigma_{d^{-1}_{i,k}(l)}(u)),
    d_{i,\sigma_j \sigma_{d^{-1}_{i,k}(l)}(u) }
    (d_{\sigma^{-1}_{j}(k) \; \sigma_{d^{-1}_{i,k }(l)}(u)}  (v))).
\end{eqnarray*}
Now condition (b) means that
\begin{eqnarray*}
    &&\sigma_j\circ\sigma_{d^{-1}_{i,k}(l)}=\sigma_l\circ\sigma_{d^{-1}_{k,i}(j)}
\end{eqnarray*}
and
\begin{eqnarray*}\lefteqn{d_{i\; \sigma_j \sigma_{d^{-1}_{i,k}(l)}(u) } \circ  d_{ \sigma^{-1}_{j}(k), \; \sigma_{d^{-1}_{i,k}(l)}(u)}}\\
    && =d_{k\; \sigma_l\sigma_{d^{-1}_{k,i}(j)}(u) } \circ  d_{\sigma^{-1}_{l}(i), \; \sigma_{d^{-1}_{k,i}(j)}(u)},
\end{eqnarray*}
for all $i,k,u\in Y$ and $j,l\in Z$. Hence condition (b) holds if and only if conditions 2. and (4') hold. Therefore the claim follows.

    Suppose that $(X,r)$ is a solution of the YBE. It is clear that $(X,r)$ is indecomposable if and only if the permutation groups $F$ and $W$ are transitive.

    Let $i,i'\in Y$ and $j,j'\in Z$ be such that
    $\sigma_{(i,j)}=\sigma_{(i',j')}$. This is equivalent to
    $$(\sigma_j(k),d_{i,\sigma_j(k)}(l))=(\sigma_{j'}(k),d_{i',\sigma_{j'}(k)}(l)),$$
    for all $k\in Y$ and $l\in Z$. Hence $(X,r)$ is irretractable if and only if
    conditions (i) and (ii) are satisfied.

    Suppose that $(X,r)$ is an irretractable solution of the YBE. By condition 2. with $j=l$, we have that $\sigma_{d^{-1}_{i,k}(j)}=\sigma_{d^{-1}_{k,i}(j)}$,
    for all $j\in Z$. Therefore $\sigma_{(u,d^{-1}_{i,k}(j))}=\sigma_{(u,d^{-1}_{k,i}(j))}$, for all $i,k,u\in Y$ and all $j\in Z$. Since $(X,r)$ is irretractable,
    we have that $d_{i,k}=d_{k,i}$ for all $i,k\in Y$, which is condition
    3. Then, by condition (4'), for
    $w=\sigma_j\sigma_{d^{-1}_{i,k}(l)}(u)$, we get
    $$d_{i,w} \circ  d_{ \sigma^{-1}_{j}(k),\sigma^{-1}_{j}(w)}
        =d_{k,w} \circ  d_{\sigma^{-1}_{l}(i), \sigma^{-1}_{l}(w)}.$$
    Thus condition 4. is satisfied.

    Hence, if $(X,r)$ is an irretractable solution of the YBE,
    conditions 2. and (4') are equivalent to conditions 2. and
    4.

    Therefore the result follows.
\end{proof}

Now, we shall construct concrete examples of indecomposable and
irretractable solutions $(X,r)$, of the form described in
Proposition~\ref{sec6newind}, with $|X|=n^2$ for every integer
$n>1$.

\begin{theorem}\label{newexample2}
    Let $n>1$ be an integer. Let $t\in\Z/(n)$ be an invertible element. Let $j_0,\dots, j_{n-1}\in
    \Z/(n)$ be elements such that $j_i=j_{-i}$ and
    \begin{equation}\label{condition} j_{t^si}=t^sj_i-(t^s-1)j_0, \end{equation}
    for all $i\in\Z/(n)$ and all $s\in\Z$. Suppose that for every
    nonzero $i\in \Z/(n)$ there exists $k\in \Z/(n)$ such that
    $j_{i+k}- j_k$ is invertible.
    Let $r\colon (\Z/(n))^2\times (\Z/(n))^2\longrightarrow
    (\Z/(n))^2\times (\Z/(n))^2$ be the map defined by
    $$r((i,j),(k,l))=(\sigma_{(i,j)}(k,l),\sigma^{-1}_{\sigma_{(i,j)}(k,l)}(i,j)),$$
    where $\sigma_{(i,j)}(k,l)=(tk+j,t(l-j_{tk+j-i}))$, for all
    $i,j,k,l\in\Z/(n)$. Then $((\Z/(n))^2,r)$ is an indecomposable and irretractable solution of
    the YBE.

    Assume, moreover, that
    \begin{itemize}
        \item[(i)] $j_0-j_i$ is invertible for every nonzero $i\in\Z/(n)$, and
        \item[(ii)] $j_i-j_k$ is invertible for all $j_i\neq j_k$.
    \end{itemize} Then  $((\Z/(n))^2,r)$ is a simple solution of
    the YBE.
\end{theorem}
\begin{proof}
    Consider the permutations $\sigma_j,d_{i,k}\in\Sym_{\Z/(n)}$ defined by
    $$\sigma_j(k)=tk+j \quad \mbox{ and }\quad d_{i,k}(l)=t(l-j_{k-i}),$$
    for all $i,j,k,l\in\Z/(n)$. Then
    $$\sigma_{(i,j)}(k,l)=(\sigma_j(k),d_{i,\sigma_j(k)}(l)),$$
    for all $i,j,k,l\in \Z/(n)$. Note that $\sigma_1\sigma^{-1}_0(k)=\sigma_1(t^{-1}k)=k+1$, for  all
    $k\in\Z/(n)$.
    Therefore the group $F=\langle\sigma_j\mid j\in\Z/(n)\rangle$ is
    transitive on $\Z/(n)$. Note also that
    $d^{-1}_{0,1+k}d_{0,k}(l)=d^{-1}_{0,1+k}(t(l-j_k))=l-j_k+j_{1+k}$ for all $k,l\in\Z/(n)$. Since
    there exists $k\in\Z/(n)$ such that $j_{1+k}-j_k$ is invertible,  the group $W=\langle d_{i,j}\mid
    i,j\in\Z/(n)\rangle$ is transitive on $\Z/(n)$. It is clear that
    condition (i) of Proposition \ref{sec6newind} is satisfied. Suppose
    that $d_{i,u}=d_{i',u}$, for all $u\in\Z/(n)$. Then
    $d_{i,u}(l)=t(l-j_{u-i})=d_{i',u}(l)=t(l-j_{u-i'})$, for all
    $u,l\in \Z/(n)$. Therefore $i=i'$ and condition (ii) of Proposition
    \ref{sec6newind} is satisfied.
Obviously condition 1. of Proposition \ref{sec6newind} is satisfied. Since $j_i=j_{-i}$, we have that
    $d_{i,k}=d_{k,i}$ and thus condition 3. of Proposition
    \ref{sec6newind} is satisfied.
    We shall check conditions 2. and 4. Note that
    \begin{eqnarray*}
            \sigma_j\sigma_{d^{-1}_{i,k}(l)}(u)&=&\sigma_j(tu+t^{-1}l+j_{i-k})\\
            &=&t(tu+t^{-1}l+j_{i-k})+j=t^2u+l+tj_{i-k}+j
        \end{eqnarray*}
        and
        \begin{eqnarray*}
            \sigma_l\sigma_{d^{-1}_{k,i}(j)}(u)&=&\sigma_l(tu+t^{-1}j+j_{k-i})\\
            &=&t(tu+t^{-1}j+j_{k-i})+l=t^2u+j+tj_{k-i}+l.
    \end{eqnarray*}
    Since $j_i=j_{-i}$, we have that $\sigma_j\circ\sigma_{d^{-1}_{i,k}(l)}=\sigma_l\circ\sigma_{d^{-1}_{k,i}(j)}$, thus condition 2. is satisfied.
    Now, (using assumption (\ref{condition}) in the fifth
    equality below), we have
  \begin{eqnarray*}
            d_{i,w}d_{\sigma^{-1}_j(k),\sigma^{-1}_j(w)}(u)&=&d_{i,w}(t(u-j_{\sigma^{-1}_j(k)-\sigma^{-1}_j(w)}))\\
            &=&d_{i,w}(t(u-j_{t^{-1}(k-w)}))\\
            &=&t(t(u-j_{t^{-1}(k-w)})-j_{i-w})\\
            &=&t^2u-t^2j_{t^{-1}(k-w)}-tj_{i-w}\\
            &=&t^2u-t^2(t^{-1}j_{k-w}-(t^{-1}-1)j_0)-tj_{i-w}\\
            &=&t^2u-tj_{k-w}+t^2(t^{-1}-1)j_0-tj_{i-w}
    \end{eqnarray*}
and thus
\begin{eqnarray*}
    d_{k,w}d_{\sigma^{-1}_l(i),\sigma^{-1}_l(w)}(u)&=&t^2u-tj_{i-w}+t^2(t^{-1}-1)j_0-tj_{k-w}.
 \end{eqnarray*}
     Hence,
     $d_{i,w}\circ
        d_{\sigma^{-1}_j(k),\sigma^{-1}_j(w)}=d_{k,w}\circ
        d_{\sigma^{-1}_l(i),\sigma^{-1}_l(w)}$,
        for all $i,j,k,l,w\in\Z/(n)$,    and thus condition 4. of
    Proposition \ref{sec6newind} is satisfied. Therefore, by
    Proposition \ref{sec6newind} $((\Z/(n))^2,r)$ is an indecomposable
    and irretractable solution of the YBE.

Assume now that conditions $(i)$ and $(ii)$ are satisfied.

Let $f: ((\Z/(n))^2,r) \rightarrow (Y,s)$ be an
epimorphism of solutions. Suppose that $f$ is not an isomorphism,
so that $|Y|<n^2$. Since $((\Z/(n))^2,r)$ is indecomposable, $(Y,s)$ also is indecomposable, and
by \cite[Lemma 1]{CCP},
$|f^{-1}(y)|=|f^{-1}(y')|$, for all $y,y'\in Y$. We write $s(y,z)=(\sigma'_y(z),\gamma'_z(y))$.

Suppose first that there exists $j\in\Z/(n)$
such that $f(i,j)=f(k,j)$ for some $i\neq k$. Then
$$f(\sigma_{(i,j)}(u,v))=f(\sigma_{(k,j)}(u,v)),$$
that is
$$f(tu+j,t(v-j_{tu+j-i}))=f(tu+j,t(v-j_{tu+j-k})),$$
for all $u,v\in\Z/(n)$. In particular, for $u=t^{-1}(i-j)$ and $v=w+j_{0}$, we get
$$f(i,tw)=f(i,t(w+j_{0}-j_{i-k})),$$
for all $w\in\Z/(n)$. Since $t$ and $j_0-j_{i-k}$ are invertible,
repeating this argument several
times, we obtain that
$$f(i,0)=f(i,w),$$
for all $w\in\Z/(n)$.  This implies that
\begin{eqnarray*}f(\sigma_{(i,0)}(i,v+j_{ti-i}))&=& \sigma'_{f(i,0)}(f(i,v+j_{ti-i}))\\
    &=&
    \sigma'_{f(i,u)}(f(i,v+j_{ti+u-i}))\\
    &=&f(\sigma_{(i,u)}(i,v+j_{ti+u-i})),\end{eqnarray*}
that is
$$f(ti,tv)=f(ti+u,tv),$$
for all $u\in\Z/(n)$,
and by the above argument,  we get that
$$f(x_{i,w})=f(x_{u,v}),$$
for all $u,v,w\in \Z/(n)$. Hence $|Y|=1$,  as desired.

Suppose that $|Y|>1$.   We have seen above that then for
every $y\in Y$, $f^{-1}(y)=\{ (i_1,k_1),\dots ,(i_m,k_m)\}$ with
$|\{k_1,\dots, k_m\}|=m$. Since $|f^{-1}(y)|=|f^{-1}(y')|=m$, for
all $y,y'\in Y$, we may assume, for a particular choice of $y$,
that $(i_1,k_1)=(0,0)$. Let $z=f(\sigma_{(0,0)}(0,0))\in Y$. Now
we have that
    $$f(\sigma_{(i_l,k_l)}(0,0))=f(\sigma_{(0,0)}(0,0))=z,$$
    for all $l=1,\dots ,m$. That is
    $$f^{-1}(z)\supseteq\{(k_l,-tj_{k_l-i_l})\mid l\in\{ 1,\dots ,m\}\}.$$
    Since $|\{k_l\mid l\in\{ 1,\dots ,m\} \}|=m=|f^{-1}(z)|$, we have that
    \begin{equation}\label{y1}f^{-1}(z)=\{(k_l,-tj_{k_l-i_l})\mid l\in\{ 1,\dots ,m\}\}.
    \end{equation}
    We also have that, for every $u\in\{ 1,\dots ,m\}$,
    $$f(\sigma_{(i_u,k_u)}(i_l,k_l))=f(\sigma_{(i_u,k_u)}(0,0))=z,$$
    for all $l=1,\dots ,m$. That is, for every $u\in\{ 1,\dots ,m\}$,
    \begin{equation}\label{y2}f^{-1}(z)=\{(ti_l+k_u,t(k_l-j_{ti_l+k_u-i_l}))\mid l\in\{ 1,\dots ,m\}\}.
    \end{equation}
    Hence
    \begin{equation}\label{x}\{k_l \mid l\in\{ 1,\dots ,m\}  \}=\{ ti_l+k_u \mid l\in\{ 1,\dots ,m\} \},
    \end{equation}
    for all $u\in\{ 1,\dots ,m\}$. In particular, for $u=1$, we have $k_1=0$ and
        $$\{k_l \mid l\in\{ 1,\dots ,m\}  \}=\{ ti_l \mid l\in\{ 1,\dots ,m\} \}.$$
Thus there  exists a permutation $\tau\in\Sym_m$ such that $ti_l=k_{\tau(l)}$, for all $ l\in\{ 1,\dots ,m\}$.
Now it is clear from (\ref{x}) that $H=\{k_l \mid l\in\{ 1,\dots ,m\}  \}$ is a subgroup of $\Z/(n)$. Furthermore the map
$$\eta:\{ i_l\mid l\in\{1,\dots ,m\}\}\longrightarrow\{k_l\mid l\in\{1,\dots, m\}\},$$
defined by $\eta(i_l)=ti_l=k_{\tau(l)}$ is an isomorphism, thus in fact
$$\{ i_l\mid l\in\{1,\dots ,m\}\}=\{ k_l\mid l\in\{1,\dots ,m\}\}$$
is the unique subgroup of $\Z/(n)$ of order $m$. Moreover, by (\ref{y1}) and (\ref{y2}) for $u=1$, we have that
$$(ti_l,t(k_l-j_{ti_l-i_l}))=(k_{\tau(l)},-tj_{k_{\tau(l)}-i_{\tau(l)}}).$$
Hence, since $t$ is invertible,
$k_l-j_{ti_l-i_l}=-j_{k_{\tau(l)}-i_{\tau(l)}}$. In particular
$$0=k_1\neq k_2=j_{ti_2-i_2}-j_{k_{\tau(2)}-i_{\tau(2)}}.$$
By condition $(ii)$, $k_2$ is invertible in $\Z/(n)$, thus its additive order is $n$.
Since $\{ k_l\mid l\in\{1,\dots ,m\}\}$ is a subgroup of order $m$ of $\Z/(n)$,
we have that $n=m$.

By (\ref{y1}), $f^{-1}(z)=\{(k_l,-tj_{k_l-i_l})\mid l\in\{ 1,\dots
,n\}\}$, and by  the description of $f^{-1}(z)$ obtained in the
first part of the proof
$$n=|\{ -tj_{k_l-i_l}\mid l\in\{ 1,\dots ,n\}\}|=|\{ j_{k_l-i_l}\mid l\in\{ 1,\dots ,n\}\}|.$$
Since $j_i=j_{-i}$, this implies that $n=2$. In this case, $t=1$, $j_0\neq j_1$ and
$$f(0,0)=f(1,1) \quad \mbox{and}\quad f(0,1)=f(1,0).$$
Furthermore, $|Y|=2$. Thus $Y=\{y_1,y_2\}$. Since $(Y,s)$ is
indecomposable, it is known that $\sigma'_{y_1}=\sigma'_{y_2}$.
Hence
$$f(1,j_1)=f(\sigma_{(0,0)}(1,0))=f(\sigma_{(1,0)}(1,0))=f(1,j_0).$$
Since $j_0\neq j_1$, we get $f(1,0)=f(1,1)$, in contradiction with
$|Y|=2$. This contradicts the assumption that $|Y|>1$. Therefore
$|Y|=1$ and the result follows.
\end{proof}

\begin{remark}\label{4simple}
Note that for $n=2$ there are two indecomposable and irretractable
solutions of the YBE constructed as in Theorem \ref{newexample2},
with $t=1$, and corresponding to either $(j_0,j_1)=(1,0)$ or
$(j_0,j_1)=(0,1)$. In this case, also conditions (i) and (ii) are
satisfied. Hence these two solutions are in fact simple. The
solution corresponding to $(j_0,j_1)=(0,1)$ is isomorphic to the
solution of Example \ref{examp1}. The solution corresponding to
$(j_0,j_1)=(1,0)$ is isomorphic to the solution of Example
\ref{examp2}.
    \end{remark}

In Theorem \ref{newexample2}, the assumptions on the elements
$j_i$ are very strong. Note that for $t=1$, condition
(\ref{condition}) is empty. In the following result we see that
for $t=1$ the other assumptions on the $j_i$ in Theorem
\ref{newexample2} can be relaxed to obtain indecomposable and
irretractable solutions and condition $(ii)$ is not needed to
obtain simple solutions.

\begin{theorem}\label{newexample2bis}
    Let $n>1$ be an integer.  Let $j_0,\dots, j_{n-1}\in
    \Z/(n)$ be elements such that $j_i=j_{-i}$,
    for all $i\in\Z/(n)$. Suppose that for every
    nonzero $i\in \Z/(n)$ there exists $k\in \Z/(n)$ such that
    $j_{i+k}\neq j_k$. Assume also that $\langle j_i\mid i\in\Z/(n)\rangle=\Z/(n)$.
    Let $r\colon (\Z/(n))^2\times (\Z/(n))^2\longrightarrow
    (\Z/(n))^2\times (\Z/(n))^2$ be the map defined by
    $r((i,j),(k,l))=(\sigma_{(i,j)}(k,l),\sigma^{-1}_{\sigma_{(i,j)}(k,l)}(i,j))$,
    where $\sigma_{(i,j)}(k,l)=(k+j,l-j_{k+j-i})$, for all
    $i,j,k,l\in\Z/(n)$. Then $((\Z/(n))^2,r)$ is an indecomposable and irretractable solution of
    the YBE.

    Assume, moreover, that
    $j_0-j_i$ is invertible for every nonzero $i\in\Z/(n)$. Then  $((\Z/(n))^2,r)$ is a simple solution of
    the YBE.
\end{theorem}
\begin{proof}
As in the proof of Theorem \ref{newexample2}, for $t=1$,
$\sigma_j(l)=l+j$ and $d_{i,k}(l)=l-j_{k-i}$,  for all $i,j,k,l\in \Z/(n)$, one can
see that $((\Z/(n))^2,r)$ is a solution of the YBE. Note that
$\langle \sigma_j \mid j\in \Z/(n)\rangle$ is a transitive
subgroup of $\Sym_{\Z/(n)}$ and, since $\langle j_i\mid
i\in\Z/(n)\rangle=\Z/(n)$, it is clear that $\langle d_{i,k}\mid
i,k\in\Z/(n)\rangle$ also is a transitive subgroup of
$\Sym_{\Z/(n)}$. Hence the solution $((\Z/(n))^2,r)$ is
indecomposable. Note that
    $$\sigma_{(i,j)}(u,v)=\sigma_{(i',j')}(u,v),$$
    for all $u,v\in\Z/(n)$, if and only if $(u+j,v-j_{u+j-i})=(u+j',v-j_{u+j'-i'})$, for all $u,v\in\Z/(n)$. This is equivalent to
    $j=j'$ and $j_{u+j-i}=j_{u+j-i'}$, for all $u\in\Z/(n)$. Since for every
    nonzero $i\in \Z/(n)$ there exists $k\in \Z/(n)$ such that
    $j_{i+k}\neq j_k$,  choosing $u=i'-j+k$ we get that $j_{u+j-i}=j_{u+j-i'}$, for all $u\in\Z/(n)$
     if and only if $i=i'$.
    Therefore the solution $((\Z/(n))^2,r)$ is irretractable.

    Assume now that
    $j_0-j_i$ is invertible for every nonzero $i\in\Z/(n)$.
    Let $f: ((\Z/(n))^2,r) \rightarrow (Y,s)$ be an
    epimorphism of solutions. Suppose that $f$ is not an isomorphism,
    so that $|Y|<n^2$. Since $((\Z/(n))^2,r)$ is indecomposable, $(Y,s)$ also is indecomposable, and
    by \cite[Lemma 1]{CCP},
    $|f^{-1}(y)|=|f^{-1}(y')|$, for all $y,y'\in Y$. We write $s(y,z)=(\sigma'_y(z),\gamma'_z(y))$.

    Suppose first that there exists $j\in\Z/(n)$
    such that $f(i,j)=f(k,j)$ for some $i\neq k$. As in the proof of Theorem \ref{newexample2}, one prove that $|Y|=1$.

    Suppose that $|Y|>1$. Hence for every $y\in Y$, $f^{-1}(y)=\{ (i_1,k_1),\dots ,(i_m,k_m)\}$
with $|\{k_1,\dots, k_m\}|=m$. Since $|f^{-1}(y)|=|f^{-1}(y')|=m$,
for all $y,y'\in Y$, we may assume, for a particular choice of $y$,
that $(i_1,k_1)=(0,0)$. Let $z=f(\sigma_{(0,0)}(0,0))\in Y$. Now we
have that
    $$f(\sigma_{(i_l,k_l)}(0,0))=f(\sigma_{(0,0)}(0,0))=z,$$
    for all $l=1,\dots ,m$. That is
    $$f^{-1}(z)\supseteq\{(k_l,-j_{k_l-i_l})\mid l\in\{ 1,\dots ,m\}\}.$$
    Since $|\{k_l\mid l\in\{ 1,\dots ,m\} \}|=m=|f^{-1}(z)|$, we have that
    \begin{equation}\label{y1bis}f^{-1}(z)=\{(k_l,-j_{k_l-i_l})\mid l\in\{ 1,\dots ,m\}\}.
    \end{equation}
In particular, by   the form of $f^{-1}(z)$ explained before,
we get that $j_0=j_{k_1-i_1}\neq j_{k_2-i_2}$. Since
$((\Z/(n))^2,r)$ is
    indecomposable, there exist $(u_1,v_1),\dots ,(u_l,v_l)\in (\Z/(n))^2$ such that
    \begin{equation}\label{ZZ}\sigma_{(u_1,v_1)}\cdots \sigma_{(u_l,v_l)}(0,-j_0)=(k_2,-j_{k_2-i_2}).
    \end{equation}
    Thus
    \begin{equation}\label{y2bis} k_2=v_1+\dots +v_l\quad \mbox{and}\quad -j_{k_2-i_2}=-j_0-j_{v_l-u_l}-j_{v_l+v_{l-1}-u_{l-1}}-\dots -j_{v_l+\dots +v_1-u_1}.
    \end{equation}
    Note that by (\ref{ZZ}) and (\ref{y2bis}),
    \begin{eqnarray}\label{lambdaY}
        \lefteqn{\sigma_{(u_1,v_1)}\cdots \sigma_{(u_l,v_l)}(v,w)}\notag\\
        &=&(v+v_1+\dots +v_l,w-j_{v_l-u_l}-j_{v_l+v_{l-1}-u_{l-1}}-\dots -j_{v_l+\dots +v_1-u_1})\notag\\
        &=&(v+k_2, w+j_0-j_{k_2-i_2})
    \end{eqnarray}
    for all $v,w\in\Z/(n)$. Since, by (\ref{y1bis}),
$f(k_2,-j_{k_2-i_2})=f(k_1,-j_{k_1-i_1})=f(0,-j_0)=z$,
 we have by (\ref{lambdaY}),
\begin{eqnarray*}
 z&=& f((\sigma_{(u_1,v_1)}\cdots \sigma_{(u_l,v_l)})^h(0,-j_0))\\
 &=&f(hk_2,-j_0 +h(j_{0}-j_{k_2-i_2}))
\end{eqnarray*}
for all positive integers $h$. Since $0=k_1-i_1\neq k_2-i_2$, by our assuption
    $j_0-j_{k_2-i_2}$ is invertible and therefore
    $$|\{-j_0+ h(j_0-j_{k_2-i_2})\mid
    h\in\Z/(n)\}|=n.$$ Hence,  again using the properties of $f^{-1}(y')$, we
have that $m=|f^{-1}(y')|\geq n$, for every $y'\in Y$. In particular, by (\ref{y1bis})
$$|\{(k_l,-j_{k_l-i_l})\mid l\in\{ 1,\dots ,n\}\}|=n=
|\{ j_{k_l-i_l}\mid l\in\{ 1,\dots ,n\}\}|.$$ Since $j_i=j_{-i}$,
this implies that $n=2$. Now the result follows as in the proof of
Theorem~\ref{newexample2}.
\end{proof}

\begin{remark}\label{exnonsimple}
    Let $n$ be an integer greater than $1$. Let $j_0,\dots ,j_{n-1}\in\Z/(n)$ be elements such that $j_i=j_{-i}$ for all $i\in\Z/(n)$.
    Suppose that the solution $(X,r)$, with $X=(\Z/(n))^2$ and $r((i,j),(k,l))=(\sigma_{(i,j)}(k,l),\sigma^{-1}_{\sigma_{(i,j)}(k,l)}(i,j))$,
    where $\sigma_{(i,j)}(k,l)=(k+j,l-j_{k+j-i})$, for all
    $i,j,k,l\in\Z/(n)$, is indecomposable and irretractable. Is $(X,r)$ a simple solution?
    The answer is negative. For example, take $n=6$, $j_0=1$ and $j_i=3$ for $i\in\Z/(6)\setminus\{ 0\}$. In this case
$$\sigma_{(i,j)}(k,l)=(k+j,l-3+2\delta_{k+j-i,0}).$$ By Theorem \ref{newexample2bis}, $((\Z/(6))^2,r)$, with
$r((i,j),(k,l))=(\sigma_{(i,j)}(k,l),\sigma^{-1}_{\sigma_{(i,j)}(k,l)}(i,j))$,
where $\sigma_{(i,j)}(k,l)=(k+j,l-j_{k+j-i})$, for all
$i,j,k,l\in\Z/(6)$, is an indecomposable and irretractable solution of the YBE. Let $Y=\{ 1,2\}$ and $s\colon
Y^2\longrightarrow Y^2$ be the map
$s(y_1,y_2)=(\sigma(y_2),\sigma(y_1))$, for all $y_1,y_2\in Y$,
where $\sigma=(1,2)\in\Sym_Y$. Thus $(Y,s)$ is a solution of the
YBE. Let $f\colon X\longrightarrow Y$ be the map defined by
$$f(x_{i,k})=\left\{ \begin{array}{ll}
1&\quad\mbox{if }k\in (2\Z)/(6),\\
2&\quad\mbox{otherwise}.
\end{array}\right.$$
It is easy to see that $f$ is an epimorphism of solutions. Hence $((\Z/(6))^2,r)$ is not simple.
\end{remark}

Leandro Vendramin has calculated all 685 irretractable solutions
of the YBE of cardinality $9$ and up to isomorphism there are only
$3$ indecomposable irretractable solutions, $(X,r_1)$ $(X,r_2)$
and $(X,r_3)$, where $X=\{1,2,3,4,5,6,7,8,9\}$ and
$r_1(x,y)=(\sigma_x(y),\sigma^{-1}_{\sigma_x(y)}(x))$, with

\begin{tabular}{ll}
    $\sigma_1=(1,6,7,9,2,5,4,8,3),\quad$ &$\sigma_2=(1,2,5,9,8,3,4,6,7)$,\\
    $\sigma_3=(1,6,5,9,2,3,4,8,7),\quad$ &$\sigma_4=(1,5,8,9,3,6,4,7,2)$,\\
    $\sigma_5=(1,3,6,9,7,2,4,5,8),\quad$ &$\sigma_6=(1,5,6,9,3,2,4,7,8)$,\\
    $\sigma_7=(1,4,9)(2,6,8),\quad$ &$\sigma_8=(1,4,9)(3,5,7)$,\\
    $\sigma_9=(2,6,8)(3,5,7)$,&\\
\end{tabular}

\noindent$r_2(x,y)=(\sigma_x(y),\sigma^{-1}_{\sigma_x(y)}(x))$, with

\begin{tabular}{ll}
    $\sigma_1=(1,6,3),\quad$ &$\sigma_2=(1,2,7,6,8,4,3,9,5)$,\\
    $\sigma_3=(1,5,9,6,7,2,3,4,8),\quad$ &$\sigma_4=(1,7,2,6,4,8,3,5,9)$,\\
    $\sigma_5=(4,5,7),\quad$ &$\sigma_6=(1,9,5,6,2,7,3,8,4)$,\\
    $\sigma_7=(1,9,7,6,2,4,3,8,5),\quad$ &$\sigma_8=(1,5,2,6,7,8,3,4,9)$,\\
    $\sigma_9=(2,8,9)$,&
\end{tabular}

\noindent and $r_3(x,y)=(\sigma_x(y),\sigma^{-1}_{\sigma_x(y)}(x))$, with

\begin{tabular}{ll}
    $\sigma_1=(1,6,3)(2,9,8)(4,7,5),\quad$ &$\sigma_2=(1,2,5,3,9,4,6,8,7)$,\\
    $\sigma_3=(1,4,2,3,7,9,6,5,8),\quad$ &$\sigma_4=(1,7,9,3,5,8,6,4,2)$,\\
    $\sigma_5=(1,3,6)(2,9,8)(4,5,7),\quad$ &$\sigma_6=(1,8,7,3,2,5,6,9,4)$,\\
    $\sigma_7=(1,8,5,3,2,4,6,9,7),\quad$ &$\sigma_8=(1,4,9,3,7,8,6,5,2)$,\\
    $\sigma_9=(1,3,6)(2,8,9)(4,7,5)$.&
\end{tabular}

\noindent There are 5 irretractable square-free solutions of the YBE of cardinality $9$.

\begin{remark}
    One can check that the three indecomposable and irretractable solutions of the YBE of cardinality $9$ are simple solutions.
    The solution $(X,r_1)$ is isomorphic to $((\Z/(3))^2,r_1')$ with $r_1'((i,j),(k,l))=(\sigma_{(i,j)}(k,l),\sigma^{-1}_{\sigma_{(i,j)}(k,l)}(i,j))$,
    where $\sigma_{(i,j)}(k,l)=(k+j,l-j_{k+j-i})$, for all
    $i,j,k,l\in\Z/(3)$, and $j_0=0$, $j_1=j_2=1$.

    The solution $(X,r_2)$ is isomorphic to $((\Z/(3))^2,r_2')$ with $r_2'((i,j),(k,l))=(\sigma_{(i,j)}(k,l),\sigma^{-1}_{\sigma_{(i,j)}(k,l)}(i,j))$,
    where $\sigma_{(i,j)}(k,l)=(k+j,l-j_{k+j-i})$, for all
    $i,j,k,l\in\Z/(3)$, and $j_0=1$, $j_1=j_2=0$.

    The solution $(X,r_3)$ is isomorphic to $((\Z/(3))^2,r_3')$ with $r_3'((i,j),(k,l))=(\sigma_{(i,j)}(k,l),\sigma^{-1}_{\sigma_{(i,j)}(k,l)}(i,j))$,
    where $\sigma_{(i,j)}(k,l)=(k+j,l-j_{k+j-i})$, for all
    $i,j,k,l\in\Z/(3)$, and $j_0=1$, $j_1=j_2=2$.
    \end{remark}

Note that all the finite simple solutions of the YBE
constructed so far have cardinality $n^2$ for some integer $n>1$.
We shall construct now finite simple solutions of the YBE of
non-square cardinality.

\begin{theorem}\label{simplemn}
Let $n,m>1$ be  integers.  Let
    $$X=\Z/(mn)\times (n\Z)/(mn)\cong \Z/(mn)\times \Z/(m).$$
Consider $\sigma_{(i,j)}\in \Sym_{X}$ defined by
$$\sigma_{(i,j)}(k,l) = \left\{ \begin{array}{ll}(k+j/n,l) &  \mbox{ if } i\neq k+j/n \\
(k+j/n,l+n ) & \mbox{ if } i=k+j/n  \end{array} \right. ,$$
for all $(i,j),(k,l)\in X$,
where $j/n$ means the unique element of the subset $\{ 0,1,\dots, m-1\}$ of $\Z/(mn)$ such that $n\cdot (j/n)=j\in\Z/(mn)$.
Let $r\colon X\times X\longrightarrow X\times X$ be the map defined by $r((i,j),(k,l))=(\sigma_{(i,j)}(k,l),\sigma^{-1}_{\sigma_{(i,j)}(k,l)}(i,j))$,
for all $(i,j),(k,l)\in X$. Then $(X,r)$ is a simple solution of the YBE.
\end{theorem}

\begin{proof} For every $j\in (n\Z)/(mn)$, let $\sigma_j\in \Sym_{\Z/(mn)}$ be the permutation defined by
    $\sigma_j(k)=k+j/n$, for all $k\in\Z/(mn)$. For $i,k\in\Z/(mn)$, we define
    $d_{i,k}\in \Sym_{(n\Z)/(mn)}$ by
    $$d_{i,k}(l)=l+n\delta_{i,k},$$
    for all $l\in (n\Z)/(mn)$. Note that $\langle\sigma_n\rangle$ is a transitive subgroup of $\Sym_{\Z/(mn)}$,
    and $\langle d_{0,0}\rangle$ is a transitive subgroup of $\Sym_{(n\Z)/(mn)}$. Also it is easy to check that
conditions of Proposition~\ref{sec6newind} are satisfied. Hence $(X,r)$ is an indecomposable and irretractable solution of the YBE.

Note that $\sigma_{(1,n)}$ is a cycle of length $nm^2$
of the form:
\begin{eqnarray} &&  (0,0)\mapsto (1,n)\mapsto  (2,n) \mapsto
    \cdots \mapsto (mn-1,n)\mapsto \nonumber\\
    &&  (0,n)\mapsto (1,2n)\mapsto (2,2n) \mapsto
    \cdots \mapsto (mn-1,2n)\mapsto \nonumber\\
    &&\cdots \mapsto \nonumber\\
    &&(0,(m-1)n)\mapsto (1,0)\mapsto (2,0) \mapsto
    \cdots \mapsto \nonumber\\
    &&(mn-1,0)\mapsto (0,0)  \label{full cycle}
\end{eqnarray}
 We know that, if
$f:(X,r)\rightarrow (T,s)$ is an epimorphism of solutions, and
$1<|T|<nm^2$, then $f^{-1}(t), t\in T$, form a system of
imprimitivity blocks in $X$ (with respect to ${\mathcal
    G}(X,r)$). By the comment after Remark~\ref{remsimple}, $(X,r)$ is a dynamical
extension of the solution $(T,s)$. So,
elements from the same block permute the set of blocks in
the same way. And there exists an integer $q>1$ such that $q| nm^2$ and each block consists of all
elements in the full cycle (\ref{full cycle}) which are in distance being a multiple
of $q$. If $q|mn$, then $mn=dq$ for some positive integer $d$ and the blocks are of the form $B_{i}=
\{ (i+qk, j) \mid k=0,1,\ldots, d-1;\; j\in (n\Z)/(mn)\}$, for $i\in\{ 0,1,\dots ,q-1\}\subseteq\Z/(mn)$. But then $\sigma_{(1,0)}(B_i)\subseteq B_i$ for every block
$B_i$ and $\sigma_{(1,n)}(B_i)\subseteq B_{i+1}$, thus
 $\sigma_{(1,0)}$ and $\sigma_{(1,n)}$ determine different permutations of the set of blocks, while
$(1,0),(1,n)\in B_1$. So $X$ cannot come from a dynamical
cocycle, a contradiction. Hence $q$ is not a divisor of $mn$. Suppose that $q<mn$. In this case, $mn=dq+t$ for some positive integers $d,t$ with $t<q$. Then the block containing $(0,0)$ is
$$B_{(0,0)}=\{ (0,0), (q,n), \dots, (dq,n), ((d+1)q,2n),  \dots\}.$$
Now $\sigma_{(0,0)}(0,0)=(0,n)\notin B_{(0,0)}$, but $\sigma_{(0,0)}(q,n)=(q,n)\in B_{(0,0)}$, a contradiction. Therefore $q>mn$.
Hence $q=dmn+t$ for some positive integers $d,t$ with $t<mn$. Then
$$B_{(0,0)}=\{ (0,0), (t,(d+1)n), \dots\}.$$
Now $\sigma_{(0,0)}(0,0)=(0,n)\notin B_{(0,0)}$, but $\sigma_{(0,0)}(t,(d+1)n)=(t,(d+1)n)\in B_{(0,0)}$, a contradiction.
It follows that
$(X,r)$ is a simple solution.
\end{proof}

\begin{remark} \label{413}
    Note that if $p_1, \dots p_k$ are distinct primes and $m_1,\dots ,m_k$ are positive integers and $m_1>1$,
    then we can take $m=p_1$ and $n=p_1^{m_1-2}p_2^{m_2}\cdots p_k^{m_k}$ in Theorem \ref{simplemn} to obtain a simple solution
    $(X,r)$ of the YBE of cardinality $m^2n=p_1^{m_1}\cdots p_k^{m_k}$.
\end{remark}

\section{Simple solutions of order $p^2$}\label{sec5}

Let $p$ be a prime number. Let $X=\Z/(p)\times\Z/(p)$. We
will study the simple solutions $(X,r)$ of the YBE, where
$r(x,y)=(\sigma_x(y),\sigma^{-1}_{\sigma_x(y)}(x))$, for all
$x,y\in X$ and
$$\sigma_{(i,j)}(k,l) =(\sigma_{j}(k),d_{i,\sigma_j(k)}(l)),$$
for all $i,j,k,l\in \Z/(p)$.  The main result in this section
shows that all the simple solutions of this form are isomorphic to
the simple solutions of cardinality $p^2$ constructed as in
Theorem \ref{newexample2}, but with weaker conditions on the
parameters $j_0,\dots ,j_{p-1}$.

\begin{theorem}\label{newexample}
Let $p$ be a prime number. Let $t\in\Z/(p)$
be a nonzero element. Let $j_0,\dots, j_{p-1}\in
\Z/(p)$ be elements such that $j_i=j_{-i}$ and
$$j_{t^si}=t^sj_i-(t^s-1)j_0,$$
for all $i\in\Z/(p)$ and all $s\in\Z$.  Suppose that $j_{i}\neq j_k$ for some $i\neq k$.
Let $r\colon (\Z/(p))^2\times (\Z/(p))^2\longrightarrow
(\Z/(p))^2\times (\Z/(p))^2$ be the map defined by
$r((i,j),(k,l))=(\sigma_{(i,j)}(k,l),\sigma^{-1}_{\sigma_{(i,j)}(k,l)}(i,j))$,
where $\sigma_{(i,j)}(k,l)=(tk+j,t(l-j_{tk+j-i}))$, for all
$i,j,k,l\in\Z/(p)$. Then $((\Z/(p))^2,r)$ is a simple solution of
the YBE.
\end{theorem}
\begin{proof} It is easy to see that since $p$ is prime and there exist $i\neq k$ such that $j_i\neq j_k$, we have that for
every nonzero $u\in \Z/(p)$ there exists $l\in \Z/(p)$ such that
$j_{u+l}-j_l$ is invertible. Hence, by the first part of Theorem
\ref{newexample2}, $((\Z/(p))^2,r)$ is an indecomposable and
irretractable solution of the YBE.

Let $f: ((\Z/(p))^2,r) \rightarrow (Y,s)$ be an epimorphism of
solutions. Suppose that $f$ is not an isomorphism and that $|Y|>1$. Since $((\Z/(p))^2,r)$ is indecomposable, $(Y,s)$ also is indecomposable,
        and  by \cite[Lemma 1]{CCP}, $|Y|=p$ and $|f^{-1}(y)|=p$, for all $y\in Y$. We write $s(y,z)=(\sigma_y(z), \gamma_z(y))$, for all $y,z\in Y$.
        It is known that $\sigma_y=\sigma_z=\sigma$ for all $y,z\in Y$, where $\sigma\in\Sym_Y$ is a cycle of length $p$  (see \cite{ESS}). The epimorphism $f$ induces an
        epimorphism of groups
        $\tilde{f}\colon\mathcal{G}((\Z/(p))^2,r)\longrightarrow \mathcal{G}(Y,s)$, such that $\tilde{f}(\sigma_{(i,j)})=\sigma_{f(i,j)}=\sigma$,
        for all $i,j\in\Z/(p)$. Furthermore, $f(\sigma_{(i,j)}(k,l))=\sigma (f(k,l))$, for all
        $i,j,k,l\in\Z/(p)$.
        Now
        $\sigma_{(k,0)}(0,0)=(0,-tj_k)$
        and
        $\sigma_{(j-k-1,j)}(0,0)=(j,-tj_{k+1})$,
        for all $j,k\in \Z/(p)$. Therefore
        $$f(0,-tj_k)=f(\sigma_{(k,0)}(0,0))=\sigma (f(0,0))=f(\sigma_{(j-k-1,j)}(0,0))=f(j,-tj_{k+1}),$$
        for all $j,k\in\Z/(p)$. Since there exists $k\in \Z/(p)$ such that $j_k\neq j_{k+1}$, we have that
        $|f^{-1}(f(0,-tj_k))|>p$, a contradiction. Therefore
        either $f$ is an isomorphism or $|Y|=1$.
    Hence, $((\Z/(p))^2,r)$ is a simple solution of the YBE.
\end{proof}

\begin{remark}
The hypothesis in the statement of Theorem \ref{newexample}
implies that the multiplicative order of $t$ is odd. Indeed if the
multiplicative order of $t$ is even, then $p$ should be odd and
there exists a positive integer $s$ such that $t^s=-1$. In this
case, the condition $j_{t^si}=t^sj_i-(t^s-1)j_0,$ implies that
$j_{-i}=-j_{i}+2j_0$. Since $j_{-i}=j_i$ and $p$ is odd, we get
that $j_i=j_0$ for all $i$, in contradiction with the hypothesis
that $j_i\neq j_k$ for some $i\neq k$.
\end{remark}

 We are now ready for the main result of this section.
\begin{theorem}\label{thmmain}
Let $p$ be a prime number. Let $X=\Z/(p)\times \Z/(p)$. Let $(X,r)$ be an indecomposable and irretractable solution of the
YBE such that $r((i,j),(k,l))=(\sigma_{i,j}(k,l),\sigma^{-1}_{\sigma_{i,j}(k,l)}(i,j))$, where
\begin{eqnarray} \label{form2} \sigma_{ij}(k,l) = (\sigma_{j}(k),d_{i, \sigma_{j}(k)}(l)),
\end{eqnarray}
for all $i,j,k,l\in\Z/(p)$. Then $(X,r)$ is isomorphic to one of the
simple solutions constructed in Theorem~\ref{newexample}.
\end{theorem}

\begin{proof} The assumptions allow to use  Proposition~\ref{sec6newind}, in particular
\begin{eqnarray}\label{cond1new}
\sigma_{j}\circ \sigma_{d^{-1}_{i,k}(l)}= \sigma_{l}\circ
\sigma_{d^{-1}_{k,i}(j)},
\end{eqnarray}
\begin{eqnarray}\label{cond3new}
    d_{i,k}=d_{k,i},
\end{eqnarray}
\begin{eqnarray}   \label{cyclenew}
    d_{i,w} \circ d_{\sigma_{j}^{-1}(k),\sigma_{j}^{-1}(w)}=
    d_{k,w} \circ d_{\sigma_{l}^{-1}(i),\sigma_{l}^{-1}(w)}.
\end{eqnarray}
    If $i=k$ then it follows that $ d_{\sigma_{j}^{-1}(i),\sigma_{j}^{-1}(w)}=
    d_{\sigma_{l}^{-1}(i),\sigma_{l}^{-1}(w)}$ for all $i,j,l,w$.
    Thus
    \begin{eqnarray}  \label{Fprimenew}
    d_{u,v} &=& d_{\sigma_{l}^{-1}\sigma_{j}(u),\sigma_{l}^{-1}\sigma_j (v)},
    \end{eqnarray}
    for all $j,l,u,v$.

By Proposition~\ref{sec6newind} we also know that $F=\langle \sigma_{j}\mid j\in \Z/(p) \rangle \subseteq \Sym_{\Z/(p)}$
and $W=\langle d_{i,k}\mid i,k\in\Z/(p)\rangle\subseteq \Sym_{\Z/(p)}$ are transitive. Let $F'=\langle \sigma_j
        \sigma_{l}^{-1} \mid l,j\in \Z/(p)\rangle \subseteq F$. By
        (\ref{cond1new})
        we know that also $\sigma_{j}^{-1}\sigma_{l}\in F'$, thus
 $F'$ is a normal subgroup of $F$. If $F'$ does not act transitively on $\Z/(p)$, then by \cite[Proposition~4.4]{passman},
        we know that the $F'$-orbits on $\Z/(p)$ are of the same cardinality. So they are singletons, which means that $F'$ acts trivially on $\Z/(p)$.
        So $F'= \{ \id\}$, in particular $\sigma_j=\sigma_k$ for all $j,k\in Z/(p)$, in contradiction with condition $(i)$ in Proposition \ref{sec6newind}.  Thus $F'$ acts transitively on $\Z/(p)$.
        Let $S_{i}$ be the stabilizer of $i\in \Z/(p)$ in $F$. Then $[F:S_{i}]=p$ and $[F':(S_{i}\cap F')]=p$. So $F'$ contains all cycles of length $p$ in $F$.
        Notice that $\sigma_{l}\in \sigma_{j}F'$, so that
        $F=\langle \sigma_{j}\rangle F'$ for every $j$.

Let $W'=\langle d_{i,j}d^{-1}_{k,l}
 \mid i,j,k,l\in \Z/(p)\rangle \subseteq W$.  By
(\ref{cyclenew})  we know that also $d_{i,j}^{-1}d_{k,l}\in W'$,
thus $W'$ is a normal subgroup of $W$. As above one can see that $W'$ acts transitively on $\Z/(p)$, using condition $(ii)$ of Proposition \ref{sec6newind}.
Notice that $d_{i,j}\in d_{k,l}W'$, so that $W=\langle d_{i,j}\rangle W'$ for every $i,j\in\Z/(p)$.
\bigskip

First we shall prove that $F$ and $W$ are abelian-by-cyclic
groups.
\bigskip

Note that condition (\ref{cond1new}) can be written in the form
    $$ \sigma_{j}\sigma_{m} = \sigma_{d_{i,k}(m)}\sigma_{d_{i,k}^{-1}(j)} \quad \mbox{ for } i,k,j,m\in \Z/(p).$$
    So for every $i,k\in\Z/(p)$, we get the permutation solution $(\Z/(p),s)$ corresponding to the permutation $d_{i,k}$, where
$s: (j,m) \mapsto (d_{i,k}(m),d_{i,k}^{-1}(j))$. Then we have a natural homomorphism of groups $G(\Z/(p),s)\rightarrow F$ such that $z\mapsto \sigma_{z}$, for all $z\in\Z/(p)$.
    But the structure group $G(\Z/(p),s)$ embeds into $\Z^{p}\rtimes \langle d_{i,k} \rangle$ (see \cite[Propositions~2.3 and~2.4]{ESS}).
    So it is abelian-by-cyclic. Hence, so is $F=\langle \sigma_{j}\mid  j\in Z \rangle\subseteq \Sym_{\Z/(p)}$.
Recall that we know that $F$ is transitive. Now, $F$ has a maximal
abelian normal subgroup $A$ such that $F/A$ is cyclic and $|F/A|$
divides the order of $d_{i,k}$, for all $i,k\in\Z/(p)$.

Suppose that $A$ is not transitive on $\Z/(p)$.
 From \cite[Proposition 4.4]{passman},
    we know that the $A$-orbits on $\Z/(p)$ are of the same cardinality. So $A= \{ \id\}$.
    Then $F$ is cyclic, so $F=\{ \id\}$ by the maximality of $A$, in contradiction with the transitivity of $F$.
Therefore $A$ is transitive on $\Z/(p)$.

    By \cite[Proposition~3.2]{passman},  we must have $A\cong \Z/(p)$.
    Since $p^2 \nmid |F|$, it follows that $p \nmid |F/A|$.
    Hence $F=A\rtimes N$ for a cyclic subgroup $N$ such that the order of $N$
    divides the order of $d_{i,k}$, for all $i,k\in \Z/(p)$.
    And $|N|$ is not divisible by $p$.

    Recall that $p\nmid |F/F'|$. Hence  $A\subseteq F'$. Let $A=\langle a\rangle$. By (\ref{Fprimenew}), we have that $d_{i,j}=d_{a^t(i),a^t(j)}$, and by (\ref{cyclenew}),
   $$d_{i,w}\circ d_{a^t\sigma_j^{-1}(k),a^t\sigma_j^{-1}(w)}=d_{k,w}\circ d_{a^t\sigma_j^{-1}(i),a^t\sigma_j^{-1}(w)},$$
for all $t\in \Z/(p)$. Since $A$ is transitive, for every $w\in
\Z/(p)$ there exists $t_j\in \Z/(p)$ such that
$a^{t_j}\sigma_j^{-1}(w)=w$. Hence
        \begin{equation}\label{newcond0}
            d_{i,w}\circ d_{a^{t_j}\sigma_j^{-1}(k),w}=d_{k,w}\circ d_{a^{t_j}\sigma_j^{-1}(i),w},
        \end{equation}
for all $i,j,k,w\in \Z/(p)$. Taking $u=a^{t_j}\sigma_j^{-1}(k)$, we have that
        $$d_{i,w}\circ d_{u,w}=d_{\sigma_ja^{-t_j}(u),w}\circ d_{a^{t_j}\sigma_j^{-1}(i),w}.$$
Hence for every the permutation $\tau_j=\sigma_ja^{-t_j}$ and $w$,
we get a permutation solution $(\Z/(p),s')$ corresponding to the
permutation $\tau_j$, where $s': (i,u) \mapsto
(\tau_j(u),\tau_j^{-1}(i))$. Since $d_{i,j}=d_{a^t(i),a^t(j)}$, it
follows that $W=\langle d_{i,w}\mid i\in\Z/(p)\rangle$. Then we
have the natural homomorphism of groups $G(\Z/(p),s')\rightarrow
W$ such that $i\mapsto d_{i,w}$, and it is surjective. But the
structure group $G(\Z/(p),s')$ embeds into $\Z^{p}\rtimes \langle
\tau_j \rangle$. So it is abelian-by-cyclic. Hence, so is $W$.
Recall that we know that $W$ is transitive. Now, $W$ has a maximal
abelian normal subgroup $A'$ such that $W/A'$ is cyclic and
$|W/A'| $ divides the order of $\tau_j$, for all $j\in \Z/(p)$.

Since $A$ is a normal subgroup of $F$, $\tau_j^m\in \sigma_j^mA$,
for all positive integers $m$. Hence $\sigma_j^m\notin A$ if and
only if $\tau_j^m\notin A$.

Note that if $\sigma_j\notin A$, then the order of $\sigma_j$ is
not divisible by $p$, because in $\Sym_{\Z/(p)}$ an element has
order divisible by $p$ if and only if it has order $p$, and $A$ is
the Sylow $p$-subgroup of $F$. Thus in this case,
$A\cap\langle\sigma_j\rangle=\{\id\}$ and
$A\cap\langle\tau_j\rangle=\{\id\}$. Since  $\tau_j^m\in
\sigma_j^mA$, for all positive integers $m$, we have that the
order of $\tau_j$ is equal to the order of $\sigma_j$, in this
case.

 On the other hand, if $\sigma_j\in A$, then since
$a^{t_j}\sigma_j^{-1}(w)=w$ and $A$ is a transitive subgroup of
order $p$ of $\Sym_{\Z/(p)}$, we have that $a^{t_j}=\sigma_j$.
Thus, by (\ref{newcond0}), we have that
$$ d_{i,w}\circ d_{k,w}=d_{k,w}\circ d_{i,w},$$
for all $i,k,w\in \Z/(p)$. Since $W=\langle d_{i,w}\mid i\in\Z/(p)\rangle$,
we have that $W$ is abelian and thus $W=A'$ in this case.

By the maximality of $A'$ and the transitivity of $W$, one can see
that $A'$ is transitive and then $A'\cong \Z/(p)$. Hence $A'$ is
the Sylow $p$-subgroup of $W$, and  it follows that $p \nmid
|W/A'|$.
    Hence $W=A'\rtimes N'$ for a cyclic subgroup $N'$ such that the order of $N'$
    divides the order of $\tau_j$, for all $j\in \Z/(p)$.
    And $|N'|$ is not divisible by $p$.
\bigskip

Now we shall see that $F=A$ if and only if $W=A'$.
\bigskip

Suppose  first that $F=A\cong  \Z/(p)$.  Thus $\sigma_j\in A$
and we have seen above that, in this case, $W=A'\cong \Z/(p)$.

Now, suppose that $W=A'\cong\Z/(p)$. Since the order of $N$
divides the order of $d_{i,k}$ and $|N|$ is not divisible by $p$,
it follows that $|N|=1$ and therefore $F=A\cong \Z/(p)$.
\bigskip

We shall study two cases.
\bigskip

{\em Case 1.} Suppose that $F=A\cong \Z/(p)\cong A'=W$.
\bigskip

In this case,  $F=\langle \sigma_t\rangle$ for some $t\in \Z/(p)$.
Hence there exists a permutation $\eta\in\Sym_{\Z/(p)}$ such that
\begin{equation} \label{sigma} \sigma_t(\eta(l))=\eta(l+1)\ \mbox{ for all }
l\in\Z/(p).
\end{equation}
 Also
there exists a permutation $\nu\in \Sym_{\Z/(p)}$ such that
$\sigma_{\nu(l)}=\sigma_t^l=\sigma_{\nu(1)}^l$, for all
$l\in\Z/(p)$. Then condition (\ref{cond1new}) implies that
\begin{eqnarray*}\sigma_{\nu(1)}^{j-l}&=&\sigma_{\nu(l)}^{-1}\circ\sigma_{\nu(j)}\\
&=& \sigma_{\nu(\nu^{-1}(d^{-1}_{k,i}(\nu(j))))}\circ
\sigma_{\nu(\nu^{-1}(d^{-1}_{i,k}(\nu(l))))}^{-1}\\
&=&\sigma_{\nu(1)}^{\nu^{-1}(d^{-1}_{k,i}(\nu(j)))-\nu^{-1}(d^{-1}_{k,i}(\nu(l)))}.\end{eqnarray*}
Hence
\begin{equation}\label{newcond1}\nu^{-1}d_{i,k}^{-1}\nu(j)-\nu^{-1}d_{i,k}^{-1}\nu(l)=j-l,
\end{equation}
for all $i,k,j,l\in\Z/(p)$. In particular,
\begin{equation}\label{condadd}\nu^{-1}d_{i,k}^{-1}\nu(j)=j+\nu^{-1}d_{i,k}^{-1}\nu(0)\
\mbox{ for all } i,k,j\in\Z/(p). \end{equation} Note that
$\sigma_{\nu(0)}=\id$. Hence condition (\ref{Fprimenew}) implies,
in view of (\ref{sigma}),  that
\begin{equation}\label{condadd2}
d_{\eta(l),\eta(t)}=d_{\sigma_{\nu(j)}(\eta(l)),\sigma_{\nu(j)}(\eta(t))}=
d_{\eta(l+j),\eta(t+j)}, \end{equation}  for all $l,t,j\in
\Z/(p)$. In particular, for $j=-t$ we have that
\begin{equation}\label{newcond2}d_{\eta(l),\eta(t)}=d_{\eta(l-t),\eta(0)},
\end{equation}
for all $l,t\in\Z/(p)$. We define
$$j_i=-\nu^{-1}d_{\eta(i),\eta(0)}^{-1}\nu(0)$$
for all $i\in\Z/(p)$. Note that by (\ref{newcond2}),
$$j_i=-\nu^{-1}d_{\eta(i),\eta(0)}^{-1}\nu(0)=-\nu^{-1}d_{\eta(0),\eta(i)}^{-1}\nu(0)=-\nu^{-1}d_{\eta(-i),\eta(0)}^{-1}\nu(0)=j_{-i},$$
for all $i\in\Z/(p)$. Since $W\cong\Z/(p)$,
$d_{\eta(i),\eta(0)}=d_{\eta(k),\eta(0)}$ if and only if
$$d_{\eta(i),\eta(0)}^{-1}\nu(0)=d_{\eta(k),\eta(0)}^{-1}\nu(0),$$
and this is equivalent to $j_i=j_k$. By condition (ii) in
Proposition~\ref{sec6newind}, if $i\neq i'$, there exists
$k\in\Z/(p)$ such that $d_{\eta (i),\eta (k)}\neq d_{\eta (
i'),\eta (k)}$. Then, by condition (\ref{newcond2}), $d_{\eta
(i-k),\eta (0)}\neq d_{\eta (i'-k),\eta (0)}$, and therefore
$j_{i-k}\neq j_{i'-k}$. Hence the elements $j_i$ satisfy the
hypothesis of Theorem~\ref{newexample} for $t=1$. Consider the
corresponding simple solution $(X,r')$, thus
$$r'((i,j),(k,l))=(\sigma'_{(i,j)}(k,l),\left(\sigma'_{\sigma'_{(i,j)}(k,l)}\right)^{-1}(i,j)),$$
where
$$\sigma'_{(i,j)}(k,l)=(k+j,l-j_{k+j-i}),$$
for all $i,j,k,l\in\Z/(p)$. We shall prove that $(X,r)$ and $(X,r')$
are isomorphic.

Let $f:X\rightarrow X$ be the map defined by
$f(i,j)=(\eta(i),\nu(j))$ for all $i,j\in \Z/(p)$.  Note that
\begin{eqnarray*}f(\sigma'_{i,j}(k,l))&=&f(k+j,l-j_{k+j-i})\\
&=&(\eta(k+j),\nu(l-j_{k+j-i}))\\
&=&(\eta(k+j),\nu(l+\nu^{-1}d_{\eta(k+j-i),\eta(0)}\nu(0)))\\
&=&(\eta(k+j),\nu(\nu^{-1}d_{\eta(k+j-i),\eta(0)}\nu(l))) \ \mbox{ by (\ref{condadd})} \\
&=&(\eta(k+j),d_{\eta(i),\eta(k+j)}\nu(l)) \ \mbox{ by  (\ref{cond3new}) and (\ref{condadd2})}\\
&=&(\sigma_{\nu(j)}(\eta(k)),d_{\eta(i),\sigma_{\nu(j)}\eta(k)}(\nu(l))) \ \mbox{ by (\ref{sigma})}\\
&=&\sigma_{(\eta(i),\nu(j))}(\eta(k),\nu(l))\\
&=&\sigma_{f(i,j)}f(k,l).
\end{eqnarray*}
Hence $f$ is an isomorphism from $(X,r')$ to $(X,r)$.
\bigskip

{\em Case 2.} Suppose that $F\neq A$ and $W\neq A'$.
\bigskip

Recall that in this case $F=A\rtimes N$, where $N$ is a cyclic
group such that $|N|>1$ is a divisor of the order of $d_{i,k}$,
for all $i,k$, and $p$ is not a divisor of $|N|$. Also
$W=A'\rtimes N'$, where $N'$ is a cyclic group such that $|N'|>1$
is a divisor of the order of $\tau_j=\sigma_ja^{-t_j}$, for every
$j\in\Z/(p)$, where $A=\langle a\rangle$. In particular,
$\tau_j\notin A$ and we have seen that in this case the
order of $\sigma_j$ is equal to the order of $\tau_j$, for all
$j\in\Z/(p)$, and $p$ is not a divisor of $|N'|$. Hence
\bigskip

  - every $d_{i,k}\neq \id$ and has order not divisible by $p$,
\bigskip

  - also every $\sigma_{j}\neq \id$ and has order prime to $p$.  \bigskip

In this case, applying condition (\ref{cond1new})
twice,
    \begin{eqnarray}\label{solutionsigma}
    \sigma_j\circ\sigma_l=\sigma_{d_{i,k}(l)}\circ\sigma_{d^{-1}_{i,k}(j)} =\sigma_{d_{u,v}d^{-1}_{i,k}(j)}\circ\sigma_{d^{-1}_{u,v}d_{i,k}(l)},\end{eqnarray}
for all $j,l\in Z$ and $i,k,u,v\in \Z/(p)$. Since  $W'$ acts
transitively on $\Z/(p)$, for every $j,l\in \Z/(p)$ there exists
$j'\in \Z/(p)$ such that
    $$\sigma_j\circ\sigma_l=\sigma_{l}\circ\sigma_{j'}.$$
In particular $A$ acts by conjugation on the set $\{ \sigma_j\mid
j\in \Z/(p)\}$. Recall that $A=\langle a\rangle\cong\Z/(p)$. Let
$j_0\in \Z/(p)$. We shall see that
$$\{ a^t\sigma_{j_0}a^{-t}\mid t\in\Z/(p)\}=\{\sigma_j\mid j\in \Z/(p)\}.$$
Note that if $\sigma_{j_0}=a^t\sigma_{j_0}a^{-t}$ for some $t\neq
0$, then $\langle a,\sigma_{j_0}\rangle$ is abelian and transitive
on $\Z/(p)$, and therefore $\sigma_{j_0}\in A$, a contradiction
because $\sigma_{j_0}\neq\id$ and has order prime to $p$.

Thus indeed $\{ a^t\sigma_{j_0}a^{-t}\mid
t\in\Z/(p)\}=\{\sigma_j\mid j\in \Z/(p)\}$. Therefore
$\sigma^{-1}_j\sigma_l=a^{t_1}\sigma^{-1}_{j_0}a^{-t_1}a^{t_2}\sigma_{j_0}a^{-t_2}$ for some
$t_1,t_2\in\Z/(p)$. Since $A$ is normal in $F$, we have that
$F'=\langle\sigma^{-1}_{j}\sigma_l\mid j,l\in Z\rangle=A$.

Since $F=F'\langle\sigma_j\rangle=A\rtimes N$, for all $j\in
\Z/(p)$, and $F'=A\cong\Z/(p)$, we have that the order of
$\sigma_j$ is equal to the order of $N$, for all $j\in\Z/(p)$,
because in this case the order of $\sigma_j$ is
relatively prime to $p$.

Recall that $W=A'\rtimes N'$, with $A'\cong\Z/(p)$ and $|N'|>1$ is
a divisor of the order of $\tau_j$ which is equal to the order of
$\sigma_j$, and $p\nmid |N'|$. Since the order of $\sigma_j$ is
equal to de order of $N$, we have that $|N'|$ divides $|N|$.
Recall that $|N|$ divides the order of any $d_{i,k}$. Since the
order of $d_{i,k}$ is not divisible by $p$, clearly the order of
$d_{i,k}$ divides $|N'|$.  Therefore $|N'|=|N|$. Let $b\in A'$ be
an element of order $p$. Thus $A'=\langle b\rangle$.

Since $A\cong A'\cong\Z/(p)$, and $N$ and $N'$ are isomorphic to
subgroups of $\Aut(\Z/(p))$ of the same order and $\Aut(\Z/(p))$
is cyclic, we have that $F\cong W\cong A\rtimes N$.

Consider the map $\varphi \colon \Z/(p)\rtimes
\Aut(\Z/(p))\longrightarrow\Sym_{\Z/(p)}$ defined by
    $$\varphi(j,g)(l)=j+g(l),$$
for all $l\in\Z/(p)$. It is clear the $\varphi$ is a monomorphism
of groups. Since $F$ is a transitive subgroup of $\Sym_{\Z/(p)}$,
it is primitive. By \cite[Theorem 2.1.6]{S}, there exist  a
subgroup $H$ of $\Aut(\Z/(p))$ and a permutation
$\eta\in\Sym_{\Z/(p)}$ such that $\varphi(\Z/(p)\rtimes
H)=\eta^{-1}F\eta$.  Since $\Aut(\Z/(p))$ is cyclic, there exists
$h\in H$ such that $H=\langle h\rangle$. Note that
$\varphi(1,\id)$ has order $p$, thus $\eta\varphi(1,\id)\eta^{-1}$
is an element of order $p$ in $F$. Hence we may assume that
$a=\eta\varphi(1,\id)\eta^{-1}$. For every $j\in\Z/(p)$, since
$\sigma_j\notin A$, there exist integers $m_j,k_j$ such that
$\eta^{-1}\sigma_j\eta=\varphi(m_j,h^{k_j})$ and $h^{k_j}\neq\id$.
Since, for every $l\in\Z/(p)$,  $\{ a^k\sigma_{l}a^{-k}\mid
k\in\Z/(p)\}=\{\sigma_j\mid j\in \Z/(p)\}$, we have that for every
$j\in\Z/(p)$ there exists $m\in\Z/(p)$ such that
$\sigma_j=a^m\sigma_0a^{-m}$. Hence
\begin{eqnarray*}\varphi(m_j,h^{k_j})&=&\eta^{-1}\sigma_j\eta =\eta^{-1}a^m\sigma_0a^{-m}\eta\\
&=&\varphi((1,\id)^m(m_0,h^{k_0})(1,\id)^{-m})=\varphi((m+m_0,h^{k_0})(-m,\id))\\
&=&\varphi(m+m_0-h^{k_0}(m),h^{k_0}).
\end{eqnarray*}
It follows that $h^{k_j}=h^{k_0}$, for all $j\in\Z/(p)$. Hence
$$\{ \eta^{-1}\sigma_j\eta\mid j\in\Z/(p)\}=\{ \varphi(l,h^{k_0})\mid l\in\Z/(p)\}.$$
Since $F=\langle\sigma_j\mid j\in\Z/(p)\rangle$ we have
\begin{eqnarray*}\varphi(\Z/(p)\rtimes\langle h\rangle)&=&\eta^{-1}F\eta\\
    &=&\langle\eta^{-1}\sigma_j\eta\mid j\in\Z/(p)\rangle\\
    &=&\varphi(\langle(l,h^{k_0})\mid l\in\Z/(p)\rangle)\\
    &=&\varphi(\Z/(p)\rtimes\langle h^{k_0}\rangle),
\end{eqnarray*}
and therefore $\langle h\rangle=\langle h^{k_0}\rangle$. Thus we
may assume that $h=h^{k_0}$. So there exists $j\in\Z/(p)$ such
that $\eta^{-1}\sigma_j\eta=\varphi(0,h)$. Since $h\in
\Aut(\Z/(p))\setminus \{\id\}$, there exists
$t\in\Z/(p)\setminus\{0,1\}$ such that $h(k)=tk$ for all
$k\in\Z/(p)$.

Now there exists a permutation $\nu$ such that
$\sigma_{\nu(k)}=a^{(1-t)^{-1}k}\sigma_{j}a^{-(1-t)^{-1}k}$.
Since $\eta^{-1}a\eta=\varphi(1,\id)$, it follows that
$\eta^{-1}a^k\eta=\varphi(k,\id)$. Therefore
\begin{eqnarray*}\eta^{-1}\sigma_{\nu(k)}\eta&=&\eta^{-1}a^{(1-t)^{-1}k}\sigma_{j}a^{-(1-t)^{-1}k}\eta\\
    &=&\varphi(((1-t)^{-1}k,\id)(0,h)(-(1-t)^{-1}k,\id))\\
&=&\varphi((1-t)^{-1}k-h((1-t)^{-1}k),h)\\
&=&\varphi((1-t)^{-1}k-t(1-t)^{-1}k,h)\\
&=&\varphi(k,h).
\end{eqnarray*}
Hence
$\eta^{-1}\sigma_{\nu(k)}\eta(l)=\varphi(k,h)(l)=k+h(l)=k+tl$, and
thus
\begin{equation}\label{condnu}
\sigma_{\nu(k)}(\eta(l))=\eta(tl+k).
\end{equation}

By (\ref{cond1new}),
$$\sigma_u\sigma_v(l)=\sigma_{d_{k,0}(v)}\sigma_{d^{-1}_{k,0}(u)}(l),$$
for all $u,v,k,l\in\Z/(p)$. In particular,
$$\sigma_{\nu(u)}\sigma_{\nu(v)}(\eta(l))=\sigma_{d_{k,0}(\nu(v))}\sigma_{d^{-1}_{k,0}(\nu(u))}(\eta(l)).$$
Hence,  applying this and condition (\ref{condnu}) several
times, we get
\begin{eqnarray*}\eta(t(tl+v)+u)&=&\sigma_{\nu(u)}(\eta(tl+v))\\
    &=&\sigma_{\nu(u)}(\sigma_{\nu(v)}(\eta(l)))\\
    &=&\sigma_{d_{k,0}(\nu(v))}\sigma_{d^{-1}_{k,0}(\nu(u))}(\eta(l))\\
    &=&\sigma_{d_{k,0}(\nu(v))}(\eta(tl+\nu^{-1}(d^{-1}_{k,0}(\nu(u)))))\\
    &=&\eta(\nu^{-1}(d_{k,0}(\nu(v)))+t(tl+\nu^{-1}(d^{-1}_{k,0}(\nu(u))))),
\end{eqnarray*}
and thus
$$tv+u=\nu^{-1}(d_{k,0}(\nu(v)))+t\nu^{-1}(d^{-1}_{k,0}(\nu(u))),$$
for all $u,v,k\in \Z/(p)$. In particular, for $u=0$ we get
$$\nu^{-1}(d_{k,0}(\nu(v)))=t(v-\nu^{-1}(d^{-1}_{k,0}(\nu(0)))).$$
Let $j_k=\nu^{-1}(d^{-1}_{\eta(k),\eta(0)}(\nu(0)))$. Hence
\begin{equation} \label{vee}
d_{\eta(k),\eta(0)}(\nu(v))=\nu(t(v-j_k)).
\end{equation}
 As $\eta^{-1}a^k\eta(u)=\varphi(k,\id)(u)=u+k$,
 by (\ref{Fprimenew}) we have
\begin{equation} \label{de}
d_{\eta(u),\eta(v)}=d_{a^k(\eta(u)),a^k(\eta(v))}=d_{\eta(u+k),\eta(v+k)},
\end{equation}
for all $k,u,v\in \Z/(p)$. Hence
$$j_k=\nu^{-1}(d^{-1}_{\eta(k),\eta(0)}(\nu(0)))=\nu^{-1}(d^{-1}_{\eta(0),\eta(-k)}(\nu(0)))=j_{-k},$$
where the last equality follows by (\ref{cond3new}).

By (\ref{cyclenew}),
$$d_{\eta(i),\eta(0)}d_{\sigma^{-1}_{\nu(j)}(\eta(k)),\sigma^{-1}_{\nu(j)}(\eta(0))}
(\nu(l))=d_{\eta(k),\eta(0)}d_{\sigma^{-1}_{\nu(j')}(\eta(i)),\sigma^{-1}_{\nu(j')}(\eta(0))}(\nu(l)),$$
for all $i,j,j',k,l\in \Z/(p)$.  Now, using (\ref{condnu}),
we get
\begin{eqnarray*}
    d_{\eta(i),\eta(0)}d_{\sigma^{-1}_{\nu(j)}(\eta(k)),\sigma^{-1}_{\nu(j)}(\eta(0))}(\nu(l))
    &=&d_{\eta(i),\eta(0)}d_{\eta(t^{-1}(k-j)),\eta(-t^{-1}j)}(\nu(l)) \ \mbox{ by (\ref{de})}\\
    &=&d_{\eta(i),\eta(0)}d_{\eta(t^{-1}k),\eta(0)}(\nu(l)) \ \mbox{ by (\ref{vee})} \\
    &=&d_{\eta(i),\eta(0)}(\nu(t(l-j_{t^{-1}k}))) \ \mbox{ by (\ref{vee})}\\
    &=&\nu(t(t(l-j_{t^{-1}k})-j_i))\quad \mbox{and similarly }\\
  d_{\eta(k),\eta(0)}d_{\sigma^{-1}_{\nu(j')}(\eta(i)),\sigma^{-1}_{\nu(j')}(\eta(0))}(\nu(l))&=&\nu(t(t(l-j_{t^{-1}i})-j_k)).
\end{eqnarray*}
Hence $tj_{t^{-1}i}+j_k=tj_{t^{-1}k}+j_i$. In particular, for
$i=0$ and $t^{-1}k=u$, we get that
$$tj_0+j_{tu}=tj_{u}+j_0.$$
Thus $j_{tu}=tj_u-(t-1)j_0$. By induction on $s$ one can see that
$$j_{t^su}=t^sj_u-(t^s-1)j_0.$$
Finally, suppose that $i$ is a nonzero element of $\Z/(p)$.
Suppose that $j_{i+k}=j_k$ for all $k\in\Z/(p)$.  Then, using
(\ref{de}) and the definition of $j_{k}$ we get
$$d_{\eta(i),\eta(-k)}=d_{\eta(i+k),\eta(0)}=d_{\eta(k),\eta(0)}=d_{\eta(0),\eta(-k)},$$
for all $k\in \Z/(p)$.
But this is not possible because $(X,r)$ is irretractable and should
satisfy condition (ii) of Proposition \ref{sec6newind}. Therefore,
for every nonzero element $i\in\Z/(p)$, there exists $k\in\Z/(p)$
such that $j_{i+k}\neq j_k$. Hence $t$ and the $j_i$ satisfy the
hypothesis of Theorem~\ref{newexample}. Consider the corresponding
simple solution $(X,r')$, with
$r((i,j),(k,l))=(\sigma'_{(i,j)}(k,l),(\sigma'_{\sigma'_{(i,j)}(k,l)})^{-1}(i,j))$,
where
$$\sigma'_{(i,j)}(k,l)=(tk+j,t(l-j_{tk+j-i})),$$
for all $i,j,k,l\in\Z/(p)$. We shall prove that the solutions
$(X,r)$ and $(X,r')$ are isomorphic. Let $f\colon X\longrightarrow
X$ be the map defined by $f(i,j)=(\eta(i),\nu(j))$ for all $i,j\in
\Z/(p)$. Note that
\begin{eqnarray*}f(\sigma'_{(i,j)}(k,l))&=&f(tk+j,t(l-j_{k+j-i}))\\
&=&(\eta(tk+j),\nu(t(l-j_{tk+j-i})))\quad\mbox{by (\ref{vee}) }\\
&=&(\eta(tk+j),\nu(\nu^{-1}d_{\eta(tk+j-i),\eta(0)}\nu(l)))\quad\mbox{by (\ref{de})  and (\ref{cond3new})}\\
&=&(\eta(tk+j),d_{\eta(i),\eta(tk+j)}\nu(l))\quad\mbox{by (\ref{condnu})}\\
&=&(\sigma_{\nu(j)}(\eta(k)),d_{\eta(i),\sigma_{\nu(j)}\eta(k)}(\nu(l)))\\
&=&\sigma_{(\eta(i),\nu(j))}(\eta(k),\nu(l))\\
&=&\sigma_{f(i,j)}f(k,l).
\end{eqnarray*}
Hence $f$ is an isomorphism from $(X,r')$ to $(X,r)$. Therefore
the result follows.
\end{proof}

\begin{remark}
Clearly, if $t=1$ in
Theorem~\ref{newexample}, then one gets many examples as the
hypothesis on the relations between $j_i$ are satisfied trivially.
 The following example shows that there are simple solutions that
satisfy the hypothesis of Theorem \ref{newexample} with $t\neq 1$.
Indeed, for  $p=7$, we have that $\{2^j\mid j\in \Z\}=\{ 2,4,1\}$.
If we take $\sigma_j(l)=2l+j$ and $d_{0,0}(l)=2l$,
$d_{1,0}(l)=2(l-1)$, $d_{2,0}(l)=2(l-2)$, $d_{4,0}(l)=2(l-4)$ and
\begin{eqnarray*}d_{k,k}=d_{0,0},&&
d_{k+1,k}=d_{k,k+1}=d_{1,0},\\
d_{k+2,k}=d_{k,k+2}=d_{2,0} &\mbox{ and
}&d_{k+4,k}=d_{k,k+4}=d_{4,0},
\end{eqnarray*}
for all $k\in\Z/(7)$, one can check that $(\Z/(7)\times
\Z/(7),r)$, where
$r((i,j),(k,l))=(\sigma_{(i,j)}(k,l),\sigma^{-1}_{\sigma_{(i,j)}(k,l)}(i,j))$
and $\sigma_{(i,j)}(k,l)=(\sigma_j(k),d_{i,\sigma_j(k)}(l))$,
satisfies the hypothesis of Theorem \ref{newexample}, and thus it
is a simple solution of the YBE.
\end{remark}

\section{An alternative construction of simple solutions of order $n^2$.}\label{section6}
In this section we shall see an alternative construction of some of
the simple solutions constructed in Theorems~\ref{newexample}
and~\ref{newexample2bis}.

 Consider the standard basis $\mathcal{B}=(e_1,\dots ,e_n)$,
    $$e_1 =(1,0,\ldots , 0),\;\; e_2 = (0,1,0,\ldots , 0), \; \ldots \; , e_{n}=(0,\ldots , 0, 1)$$
    of the  free $\Z/(n)$-module  $\left( \Z/(n) \right)^n$.
    The indices are viewed as elements of
    $\Z/(n)$.
    Consider the wreath product
    $$\Z/(n) \wr S = \left( \Z/(n) \right)^{n} \rtimes_{\alpha} \Z/(n),$$
    where
    $$\alpha(i)(e_j)=e_{i+j},$$
    for all $i,j\in \Z/(n)$.
    For $j_0,j_1,\dots ,j_{n-1}\in \Z/(n)$, such that $j_i=j_{-i}$
 for all $i\in\Z/(n)$
    (interpreting the indices as elements of $\Z/(n)$),
     consider the bilinear form
    $$b_{j_0,\dots, j_{n-1}}: \left( \Z/(n)\right)^{n} \times \left( \Z/(n) \right)^{n} \rightarrow \Z/(n),$$
    with matrix with respect to $\mathcal{B}$
    $$M(b_{j_0,\dots,j_{n-1}},\mathcal{B})=\left(
    \begin{array}{ccccc}
    j_0&j_1&\ldots&\ldots&j_{n-1}\\
    j_{n-1}&j_0&\ddots&&\vdots\\
    \vdots&\ddots&\ddots&\ddots&\vdots\\
    \vdots&&\ddots&\ddots&j_1\\
    j_1&j_2&\ldots&j_{n-1}&j_0
    \end{array}\right).$$
    Note that $\alpha(i)$ is in the orthogonal group of $((\Z/(n))^n,b_{j_0,\dots,j_{n-1}})$ for all $i\in \Z/(n)$.
    Consider the asymmetric product  (\cite{CCS}) of the trivial braces $\left(\Z/(n)\right)^{n}$ and $\Z/(n)$ defined via $\alpha$ and $b_{j_0,\dots ,j_{n-1}}$:
    $$B_{j_0,\dots, j_{n-1}} = \left( \Z/(n)\right)^{n} \rtimes_{\circ} \Z/(n).$$
    Hence, $(B_{j_0,\dots ,j_{n-1}},\circ) \cong \left( \Z/(n) \right)^{n} \rtimes_{\alpha} \Z/(n)$ and the addition is defined by
    $$(u,i)+(v,j) =(u+v, i+j+b_{j_0,\dots ,j_{n-1}}(u,v)),$$
    for all $u,v \in  \left( \Z/(n) \right)^{n}$ and $i,j \in \Z/(n)$. We know that $(B_{j_0,\dots ,j_{n-1}},+,\circ)$ is a left brace.
    We denote the lambda map of $B_{j_0,\dots ,j_{n-1}}$ by $\lambda^{(j_0,\dots,j_{n-1})}$. We have
    \begin{eqnarray}
    \lambda^{(j_0,\dots ,j_{n-1})}_{(u,i)}(v,j) &=& (u,i) (v,j) -(u,i) \nonumber \\
    &=& (u+\alpha(i)(v), i+j)-(u,i) \nonumber \\
    &=& (\alpha (i)(v),j-b_{j_0,\dots ,j_{n-1}}(u,\alpha (i)(v))) . \label{lambdaui}
    \end{eqnarray}
    Note that $\lambda^{(j_0,\dots ,j_{n-1})}_{(u,i)}$ is determined by the action on $X=\{ (e_i,j) \mid i,j \in \Z/(n)\}$.
    Put
    $$x_{ij} =(e_i,j).$$
    Note that
    \begin{eqnarray} \label{action1}
    \lambda^{(j_0,\dots, j_{n-1})}_{x_{ij}} (x_{kl}) &=& \lambda_{(e_i,j)} (e_{k},l) = (e_{k+j},l-a_{k+j-i})
    = x_{k+j, l-j_{k+j-i}}.
    \end{eqnarray}
    Let
    $$r_{j_0,\dots ,j_{n-1}}:X\times X \rightarrow X\times X$$
    be the map defined by
    $$r_{j_0,\dots,j_{n-1}}(x_{ij}, x_{kl}) = (\lambda^{(j_0,\dots,j_{n-1})}_{x_{ij}} (x_{kl}),
    (\lambda^{(j_0,\dots,j_{n-1})}_{\lambda^{(j_0,\dots,j_{n-1})}_{x_{ij}}(x_{kl})})^{-1} (x_{ij} )). $$
    Note that $r_{j_0,\dots,j_{n-1}}$ is the restriction  to $X^{2}$ of the solution associated
    to the left brace $B_{j_0,\dots ,j_{n-1}}$ (see \cite[Lemma 2]{CJOComm}).
    Thus $(X,r_{j_0,\dots,j_{n-1}})$ is a solution of the YBE.
\bigskip

     Consider the map $f\colon (\Z(n))^2\rightarrow X$ defined by $f(i,j)=x_{ij}$, for all $i,j\in\Z/(n)$. Then, by (\ref{action1}),  $f$ is an isomorphism of the solution $((\Z(n))^2,r)$ to $(X,r_{j_0,\dots,j_{n-1}})$, where $r$ is defined as in Theorem \ref{newexample2bis}, that is
     $r_{(i,j)}(k,l)=(\sigma_{(i,j)}(k,l),\sigma^{-1}_{\sigma_{(i,j)}(k,l)})$, where $\sigma_{(i,j)}(k,l)=(k+j,l-j_{k+j-i})$, for all $i,j,k,l\in\Z/(n)$. Thus the following result is an easy consequence.

    \begin{proposition}\label{ind}
        The solution $(X,r_{j_0,\dots,j_{n-1}})$ is indecomposable if and only if $\Z/(n) =\langle j_0,\dots ,j_{n-1}\rangle$.

        The solution  $(X,r_{j_0,\dots,j_{n-1}})$ is irretractable if and only if
        the $n$ rows of $M(b_{j_0,\dots,j_{n-1}},\mathcal{B})$ are distinct (i.e. if $j_{i+k}=j_{i'+k}$ for every $k$ then $i=i'$).
    \end{proposition}

    \begin{proposition}\label{Bc}
        The left braces $B_{j_0,\dots,j_{n-1}}$ and $\mathcal{G}(X,r_{j_0,\dots,j_{n-1}})$ are isomorphic if and only if $b_{j_0,\dots,j_{n-1}}$ is non-singular.
    \end{proposition}
    \begin{proof}
        First we will show that $\soc(B_{j_0,\dots ,j_{n-1}})=\{ (0,0)\}$  if and only if $b_{j_0,\dots,j_{n-1}}$ is non-singular.
 Let $(u,i)\in \soc(B_{j_0,\dots ,j_{n-1}})$. By (\ref{lambdaui}) and the definition of $\alpha$, we have that $i=0$ and $b_{j_0,\dots,j_{n-1}}(u,v)=0$,
for all $v\in (\Z/(n))^n$. Thus   $\soc(B_{j_0,\dots ,j_{n-1}})=\{
(0,0)\}$ if and only if $b_{j_0,\dots,j_{n-1}}$ is non-singular.

Let
$(B_{j_0,\dots,j_{n-1}},r_{B_{j_0,\dots,j_{n-1}}})$ be the solution
associated to the left brace $B_{j_0,\dots,j_{n-1}}$, that is,
$r_{B_{j_0,\dots,j_{n-1}}}\colon B_{j_0,\dots,j_{n-1}}\times
B_{j_0,\dots,j_{n-1}}\longrightarrow B_{j_0,\dots,j_{n-1}}\times
B_{j_0,\dots,j_{n-1}}$ is defined by
        $$r_{B_{j_0,\dots,j_{n-1}}}((u,i),(v,j)) = (\lambda^{(j_0,\dots,j_{n-1})}_{(u,i)} (v,j),
        (\lambda^{(j_0,\dots,j_{n-1})}_{\lambda^{(j_0,\dots ,j_{n-1})}_{(u,i)}(v,j)})^{-1} (u,i)),$$
        for all $u,v\in (\Z/(n))^{n}$ and $i,j\in \Z/(n)$.
        By Lemma \ref{known}, $$B_{j_0,\dots,j_{n-1}}/\soc(B_{j_0,\dots,j_{n-1}})\cong\mathcal{G}(B_{j_0,\dots,j_{n-1}},r_{B_{j_0,\dots,j_{n-1}}})$$
        as left braces. Note that the map
       $$\varphi\colon \mathcal{G}(B_{j_0,\dots,j_{n-1}},r_{B_{j_0,\dots,j_{n-1}}})\longrightarrow \mathcal{G}(X,r_{j_0,\dots,j_{n-1}})$$
        defined by
        $\varphi(\lambda^{(j_0,\dots,j_{n-1})}_{(v,j)})= \lambda^{(j_0,\dots,j_{n-1})}_{(v,j)}|_{X}$,
for all $(v,j)\in B_{j_0,\dots,j_{n-1}}$, is an isomorphism of the
multiplicative groups.

Furthermore, the inclusion map
$X\longrightarrow B_{j_0,\dots,j_{n-1}}$ is a homomorphism of
solutions from $(X,r_{j_0,\dots,j_{n-1}})$ to
$(B_{j_0,\dots,j_{n-1}},r_{B_{j_0,\dots,j_{n-1}}})$. Hence it
induces a homomorphism of left braces
$$\mathcal{G}(X,r_{j_0,\dots,j_{n-1}})\longrightarrow
\mathcal{G}(B_{j_0,\dots,j_{n-1}},r_{B_{j_0,\dots,j_{n-1}}}),$$
        which is exactly $\varphi^{-1}$. Therefore the result follows.
    \end{proof}

    \begin{proposition}\label{iso}
        Let $a_0,\dots, a_{n-1},c_0,\dots ,c_{n-1}\in \Z/(n)$ be elements such that $a_i=a_{-i}$ and $c_i=c_{-i}$ for all $i\in \Z/(n)$.
    If there exists an invertible element $a\in \Z/(n)$ such that $aa_i=c_{ai}$ for all $i\in\Z/(n)$, then the solutions $$(X,r_{a_0,\dots,a_{n-1}})\mbox{ and }(X,r_{c_0,\dots,c_{n-1}})$$
        are isomorphic. Furthermore, if $b_{a_0,\dots,a_{n-1}}$ and $b_{c_0,\dots,c_{n-1}}$ are non-singular, then the converse holds.
    \end{proposition}

    \begin{proof}
        Suppose that there exists an invertible element $a\in \Z/(n)$ such that $aa_{i}=c_{ai}$, for all $i\in\Z/(n)$.
        We define $h\colon X\longrightarrow X$ by $h(x_{ij})=x_{ai,aj}$, for all $i,j\in\Z/(n)$. We have
        $$h(\lambda^{(a_0,\dots,a_{n-1})}_{x_{i,j}}(x_{k,l}))=h(x_{k+j,l-a_{k+j-i}})=x_{a(k+j),a(l-a_{k+j-i})}$$
        and
        \begin{eqnarray*}\lambda^{(c_0,\dots,c_{n-1})}_{h(x_{i,j})}(h(x_{k,l}))
            &=&\lambda^{(c_0,\dots,c_{n-1})}_{x_{ai,aj}}(x_{ak,al})\\
            &=&x_{ak+aj, al-c_{ak+aj-ai}}=x_{a(k+j),a(l-a_{k+j-i})},
        \end{eqnarray*}
        for all $i,j,k,l\in \Z/(n)$. Therefore $h$ is an isomorphism of solutions from $(X,r_{a_0,\dots,a_{n-1}})$ to $(X,r_{c_0,\dots,c_{n-1}})$.

        Suppose that $b_{a_0,\dots,a_{n-1}}$ and $b_{c_0,\dots,c_{n-1}}$ are non-singular,
    and $(X,r_{a_0,\dots,a_{n-1}})$ and $(X,r_{c_0,\dots,c_{n-1}})$ are isomorphic.
Let $g\colon X\longrightarrow X$ be an isomorphism of solutions from
$(X,r_{a_0,\dots,a_{n-1}})$ to $(X,r_{c_0,\dots,c_{n-1}})$. Then $g$
induces an isomorphism of left braces
        $$\tilde{g}\colon\mathcal{G}(X,r_{a_0,\dots ,a_{n-1}}) \longrightarrow\mathcal{G}(X,r_{c_0,\dots,c_{n-1}}),$$
        such that $\tilde{g}(\lambda^{(a_0,\dots,a_{n-1})}_{x_{i,j}})=\lambda^{(c_0,\dots,c_{n-1})}_{g(x_{i,j})}$,
for all $x_{i,j}\in X$. Since $b_{a_0,\dots,a_{n-1}}$ and
$b_{c_0,\dots,c_{n-1}}$ are non-singular, by Proposition
\ref{Bc}, and its proof $g$ induces an isomorphism  of left
braces from $B_{a_0,\dots,a_{n-1}}$ to $B_{c_0,\dots,c_{n-1}}$
that we also denote by $g$. Note the multiplicative groups of
$B_{a_0,\dots ,a_{n-1}}$ and $B_{c_0,\dots,c_{n-1}}$ are the same,
that is $(\Z/(n))^n\rtimes_{\alpha}\Z/(n)$. In particular, $g$ is
an automorphism of this group. In this group we have
        $$(u,k)x_{i,0}(u,k)^{-1}=(u,k)(e_i,0)(-\alpha(-k)(u),-k)=(e_{i+k},0)=x_{i+k,0},$$
        for all $u\in (\Z/(n))^n$ and $i,k\in \Z/(n)$.
Hence $\{ x_{i,0}\mid i\in \Z/(n)\}$ is a conjugacy class contained
in $X$. We shall see that if $n>2$, then this is the only conjugacy
class of $(\Z/(n))^n\rtimes_{\alpha}\Z/(n)$ contained in $X$. So,
assume that $n>2$. Let $j\in \Z/(n)$ be a nonzero element. We have
        \begin{eqnarray*}x_{i+j,j}x_{i,j}x_{i+j,j}^{-1}&=&(e_{i+j},j)(e_i,j)(-e_{i},-j)\\
            &=&(2e_{i+j},2j)(-e_{i},-j)\\
            &=&(2e_{i+j}-e_{i+2j},j)\notin X,
        \end{eqnarray*}
        for all $i\in \Z/(n)$, because $2\neq 0$ in $\Z/(n)$.
Hence $\{ x_{i,0}\mid i\in \Z/(n)\}$ indeed is the only conjugacy
class of $(\Z/(n))^n\rtimes_{\alpha}\Z/(n)$ contained in $X$. Since
$g(X)=X$ and $g(C)$ is a conjugacy class of
$(\Z/(n))^n\rtimes_{\alpha}\Z/(n)$, for every conjugacy class $C$,
we have that $g(\{ x_{i,0}\mid i\in \Z/(n)\})\subseteq X$ is a
conjugacy class of $(\Z/(n))^n\rtimes_{\alpha}\Z/(n)$. Thus, we have
that
        $$g(\{x_{i,0}\mid i\in\Z/(n)\})=\{x_{i,0}\mid i\in\Z/(n)\}.$$
        Hence there exists $\nu\in\Sym_{\Z/(n)}$ such that $g(x_{i,0})=x_{\nu(i),0}$, for all $i\in \Z/(n)$.
Let $g(0,1)=(u,a)$, for some $u\in(\Z/(n))^{n}$ and $a\in \Z/(n)$.
Let $g(x_{1,1})=x_{k,l}$. Since
$x_{1,0}^{-1}x_{1,1}=(-e_1,0)(e_1,1)=(0,1)$ in
$B_{a_0,\dots,a_{n-1}}$, we have that
$(u,a)=g(x_{1,0})^{-1}g(x_{1,1})=x_{\nu(1),0}^{-1}x_{k,l}=(e_{k}-e_{\nu(1)},l)$
in $B_{c_0,\dots,c_{n-1}}$. Hence $u=e_k-e_{\nu(1)}$ and $a=l$.
Note that
        \begin{eqnarray*}g(x_{2,1})&=&g((e_{2},0)+(0,1))=g(e_{2},0)+g(0,1)\\
            &=&g(x_{2,0})+g(0,1)=x_{\nu(2),0}+(e_k-e_{\nu(1)},a)\\
            &=&(e_{\nu(2)}+e_k-e_{\nu(1)},a+b_{c_0,\dots,c_{n-1}}(e_{\nu(2)},e_k-e_{\nu(1)}))\in X.
        \end{eqnarray*}
        Hence $k=\nu(1)$, and thus $g(0,1)=(0,a)$ and
        \begin{eqnarray*}g(x_{i,j})&=&g((e_i,0)+j(0,1))=g(e_i,0)+jg(0,1)=g(x_{i,0})+j(0,a)\\
            &=&x_{\nu(i),0}+(0,ja)=x_{\nu(i),ja},
        \end{eqnarray*}
        for all $i,j\in\Z/(n)$.
        Since the additive order of $(0,1)$ in $B_{a_0,\dots,a_{n-1}}$ is $n$,
        the additive order of $(0,a)=g(0,1)$ in $B_{c_0,\dots,c_{n-1}}$ is $n$. Thus $a\in\Z/(n)$ is invertible.  Now, by (\ref{action1}),
        $$g(\lambda^{(a_0,\dots ,a_{n-1})}_{x_{i,j}}(x_{k,l}))=g(x_{k+j,l-a_{k+j-i}})=x_{\nu(k+j),a(l-a_{k+j-i})}$$
        and
        $$\lambda^{(c_0,\dots,c_{n-1})}_{g(x_{i,j})}(g(x_{k,l}))
        =\lambda^{(c_0,\dots,c_{n-1})}_{x_{\nu(i),aj}}(x_{\nu(k),al})=x_{\nu(k)+aj, al-c_{\nu(k)+aj-\nu(i)}}.$$
        Since $g(\lambda^{(a_0,\dots,a_{n-1})}_{x_{i,j}}(x_{k,l}))=\lambda^{(c_0,\dots,c_{n-1})}_{g(x_{i,j})}(g(x_{k,l}))$, we have that  $\nu(k+j)=\nu(k)+aj$, for all $k,j\in\Z/(n)$ and
        $$aa_{k+j-i}=c_{\nu(k)+aj-\nu(i)}=c_{\nu(i+k-i)+aj-\nu(i)}=c_{a(k-i)+aj},$$
        as desired.

        Finally, for $n=2$, there are only two solutions  $(X,r_{0,1})$ and $(X,r_{1,0})$ with $b_{0,1}$ and $b_{1,0}$ non-singular.
The solution $(X,r_{0,1})$ is isomorphic to the solution of Example
\ref{examp1}, and the solution $(X,r_{1,0})$ is isomorphic to the
solution of Example \ref{examp2}. Hence $(X,r_{0,1})$ and
$(X,r_{1,0})$ are not isomorphic.

        Thus the result follows.
    \end{proof}

\section{Open Questions} \label{sec7}
To conclude, we propose some open questions related to the
results obtained in the previous sections.

In Sections \ref{examples}, \ref{sec5} and \ref{section6} we study
finite indecomposable and irretractable solutions $(X,r)$ of the
YBE, where $X=Y\times Z$, the $X_y=\{ (y,z)\mid z\in Z\}$, for
$y\in Y$, are blocks of imprimitivity for the action of
$\mathcal{G}(X,r)$ on $X$, and
$r((y,z),(y',z'))=(\sigma_{(y,z)}(y',z'),\sigma^{-1}_{\sigma_{(y,z)}(y',z')}(y,z))$,
where
$\sigma_{(y,z)}(y',z')=(\sigma_z(y'),d_{y,\sigma_z(y')}(z'))$, for
some permutations $\sigma_z\in\Sym_Y$ and $d_{y,y'}\in\Sym_{Z}$,
for all $y,y'\in Y$ and $z,z'\in Z$.

\begin{question}\label{quest1}
    Let $Y,Z$ be finite sets of cardinality $>1$. Let $X=Y\times Z$. Study the indecomposable
    and irretractable solutions $(X,r)$, where the sets $X_y=\{ (y,z)\mid z\in Z\}$, for $y\in Y$,
    are blocks of imprimitivity for the action of $\mathcal{G}(X,r)$ on $X$, and
    $r((y,z),(y',z'))=(\sigma_{(y,z)}(y',z'),\sigma^{-1}_{\sigma_{(y,z)}(y',z')}(y,z))$, where
    $\sigma_{(y,z)}(y',z')=(\beta_{y,z}(y'),\alpha_{y,z,y'}(z'))$, for some permutations
    $\beta_{y,z}\in\Sym_Y$ and $\alpha_{y,z,y'}\in\Sym_{Z}$, for all $y,y'\in Y$ and $z,z'\in Z$.
\end{question}

\begin{question}
    With the same notation as in Question \ref{quest1}, assume that $|Y|=|Z|=p$ is prime.
    Suppose that $(X,r)$ is an indecomposable and irretractable solution of the YBE.
    Is it true that $\beta_{y,z}=\beta_{y',z}$, for all $y,y'\in Y$ and all $z\in Z$?
    Is it true that $\alpha_{y,z,y'}=\alpha_{y,z',y''}$ if and only if $\beta_{y,z}(y')=\beta_{y,z'}(y'')$?
\end{question}

\begin{question}
    Let $p$ be a prime. Let $(X,r)$ be an indecomposable and irretractable solution of the YBE of cardinality $p^2$.
    Is $(X,r)$ simple? Is it isomorphic to one of the simple solutions constructed in Theorem~\ref{newexample}?
\end{question}

\begin{question}
    Consider the indecomposable and irretractable solutions constructed in Theorem \ref{newexample2bis}.
    Can they be constructed using the asymmetric product of left braces?
\end{question}

Note that for $t=1$ we give such a construction in Section \ref{section6}.

\begin{question}
    Does there exist a simple solution $(X,r)$ of the YBE
    such that $|X|=p_1p_2\cdots p_n$, for $n>1$ distinct primes $p_1,p_2,\dots
    ,p_n$?
\end{question}
 Note that if $p_1,p_2,\dots, p_n$ are $n$ distinct primes, $k_1,\dots ,k_n$ are positive
 integers and $n>1$, then by Remark~\ref{413},
 there exist simple solutions of the YBE of cardinality $p_1^{k_1}p_2^{k_2}\cdots p_n^{k_n}$
 if $\sum_{i=1}^nk_{i}>n$.

 Note also that the finite simple solutions $(X,r)$ of the YBE, with $X=Y\times Z$,
 constructed in Sections \ref{examples}, \ref{sec5} and \ref{section6} satisfy that
 $1<|Y|,|Z|$ and $|Z|$ is a divisor of $|Y|$. Thus it seems natural to ask the following
 question.

\begin{question}
    Does there exist a finite simple solution $(X,r)$ of the YBE such that $X=Y\times Z$,
    the sets $X_y=\{(y,z)\mid z\in Z\}$, for $y\in Y$,
    are blocks of imprimitivity for the action of $\mathcal{G}(X,r)$ on $X$, and $|Z|$ is not a divisor of $|Y|$?
\end{question}

\section*{Acknowledgements}

The authors thank Eric Jespers for his comments and suggestions,
and Leandro Vendramin for the information about the indecomposable
and irretractable solutions of cardinality $9$.

\vspace{30pt}
 \noindent \begin{tabular}{llllllll}
  F. Ced\'o && J. Okni\'{n}ski \\
 Departament de Matem\`atiques &&  Institute of
Mathematics \\
 Universitat Aut\`onoma de Barcelona &&   Warsaw University \\
08193 Bellaterra (Barcelona), Spain    &&  Banacha 2, 02-097 Warsaw, Poland \\
 cedo@mat.uab.cat && okninski@mimuw.edu.pl\\

\end{tabular}

\end{document}